\documentclass[10pt,reqno]{amsart}

\usepackage{amssymb,latexsym,mathrsfs,amsmath}
\usepackage{amsthm}
\usepackage{graphicx}
\usepackage{tikz}
\usepackage{caption}
\usepackage{subcaption}
\usepackage{etex}

\usepackage[all,2cell]{xy}
\SelectTips{eu}{12}

\newcommand{\Z}{\mathbb{Z}}

\newcommand{\Q}{\mathbb{Q}}
\newcommand{\R}{R_3}

\newcommand{\CK}{C_{\mathrm{Kh}}}
\newcommand{\co}{\,\#\,}
\newcommand{\Cobd}{\mathcal{C}ob^3_{\bullet}}
\newcommand{\Cob}{\mathcal{C}ob^3_{\bullet/l}}

\newcommand{\smoothingud}[3]{
\draw (#1,#2) to [out = 45, in = 315] (#1,#2+#3);
\draw (#1+#3,#2) to [out = 135, in = 225] (#1+#3,#2+#3);
}

\newcommand{\smoothinglr}[3]{
\draw(#1,#2) to [out = 45, in = 135] (#1+#3,#2);
\draw(#1,#2+#3) to [out = 315, in = 225] (#1+#3,#2+#3);
}

\newcommand{\sphere}[3]{
\shade[ball color = gray!40, opacity = 0.3] ({#1},{#2}) circle ({#3});
\draw (#1,#2) circle ({#3});
\draw (#1-#3,#2) arc (180:360:#3 and 0.3*#3);
\draw[dashed] (#1+#3,#2) arc (0:180:#3 and 0.3*#3);
}

\newcommand{\spheredot}[3]{
\sphere{#1}{#2}{#3}
\node at (#1-0.3*#3,#2+0.45*#3) [scale = 1.5*#3] {$\bullet$};
}

\newcommand{\plane}[2]{
\shade[color = gray!40, opacity = 0.3] (#1,#2-0.5) -- (#1+1,#2) -- (#1+1,#2+1.5) -- (#1,#2+1) -- (#1,#2-0.5);
\draw (#1,#2-0.5) -- (#1+1,#2) -- (#1+1,#2+1.5) -- (#1,#2+1) -- (#1,#2-0.5);
}

\newcommand{\cylinder}[4]{
\shade[ball color = gray!40, opacity = 0.3] (#1,#2) arc (180:360:0.4*#3 and 0.2*#3) -- (#1+0.8*#3,#2+#4*#3) 
arc (0:180:0.4*#3 and -0.2*#3);
\shade[ball color = gray!40, opacity = 0.15] (#1,#2+#4*#3) arc (180:360:0.4*#3 and -0.2*#3) arc (0:180:0.4*#3 and -0.2*#3);
\draw (#1,#2) arc (180:360:0.4*#3 and 0.2*#3) -- (#1+0.8*#3,#2+#4*#3) 
arc (0:180:0.4*#3 and 0.2*#3) -- (#1,#2);
\draw[dashed] (#1+0.8*#3,#2) arc (0:180:0.4*#3 and 0.2*#3);
\draw (#1,#2+#4*#3) arc (180:360:0.4*#3 and 0.2*#3);
}

\newcommand{\bowl}[3]{
\shade[ball color = gray!40, opacity = 0.3] (#1+0.8*#3,#2) arc (0:180:0.4*#3 and -0.2*#3) -- (#1,#2-0.1*#3) arc (180:360:0.4*#3 and 0.4*#3) -- (#1+0.8*#3,#2);
\shade[ball color = gray!40, opacity = 0.15] (#1,#2) arc (180:360:0.4*#3 and -0.2*#3) arc (0:180:0.4*#3 and -0.2*#3);
\draw (#1+0.8*#3,#2) arc (0:180:0.4*#3 and -0.2*#3) -- (#1,#2-0.1*#3) arc (180:360:0.4*#3 and 0.4*#3) -- (#1+0.8*#3,#2);
\draw (#1,#2) arc (180:360:0.4*#3 and -0.2*#3);
}

\newcommand{\bowldot}[3]{
\bowl{#1}{#2}{#3}
\node at (#1+0.3*#3,#2-0.3*#3) [scale = 0.7*#3] {$\bullet$};
}

\newcommand{\bowlud}[3]{
\shade[ball color = gray!40, opacity = 0.3] (#1,#2) arc (180:360:0.4*#3 and 0.2*#3) -- (#1+0.8*#3,#2+0.1*#3) arc (0:180:0.4*#3 and 0.4*#3) -- (#1,#2);
\draw (#1,#2) arc (180:360:0.4*#3 and 0.2*#3) -- (#1+0.8*#3,#2+0.1*#3) arc (0:180:0.4*#3 and 0.4*#3) -- (#1,#2);
\draw[dashed] (#1+0.8*#3,#2) arc (0:180:0.4*#3 and 0.2*#3);
}

\newcommand{\bowluddot}[3]{
\bowlud{#1}{#2}{#3}
\node at (#1+0.3*#3,#2+0.3*#3) [scale = 0.7*#3] {$\bullet$};
}

\newcommand{\strand}[3]{
\draw (#1,#2) -- (#1,#2+#3);
}

\newcommand{\fstrand}[4]{
\strand{#1}{#2}{#4}
\strand{#1+0.5*#3}{#2}{#4}
\strand{#1+#3}{#2}{#4}
\strand{#1+1.5*#3}{#2}{#4}
}

\newcommand{\tstrand}[4]{
\strand{#1}{#2}{#4}
\strand{#1+0.5*#3}{#2}{#4}
\strand{#1+#3}{#2}{#4}
}

\newcommand{\jcap}[4]{
\draw (#1,#2) to [out=90,in=180] (#1+0.25*#3,#2+0.25*#4) to [out=0,in=90] (#1+0.5*#3,#2);
}

\newcommand{\jcup}[4]{
\draw (#1,#2+#4) to [out=270,in=180] (#1+0.25*#3,#2+0.75*#4) to [out=0,in=270] (#1+0.5*#3,#2+#4);
}

\newcommand{\capcup}[4]{
\jcap{#1}{#2}{#3}{#4}
\jcup{#1}{#2}{#3}{#4}
}

\newcommand{\dcapcup}[4]{
\capcup{#1}{#2}{#3}{#4}
\capcup{#1+#3}{#2}{#3}{#4}
}

\newcommand{\lcapcup}[4]{
\capcup{#1}{#2}{#3}{#4}
\strand{#1+#3}{#2}{#4}
\strand{#1+1.5*#3}{#2}{#4}
}

\newcommand{\mcapcup}[4]{
\strand{#1}{#2}{#4}
\capcup{#1+0.5*#3}{#2}{#3}{#4}
\strand{#1+1.5*#3}{#2}{#4}
}

\newcommand{\rcapcup}[4]{
\strand{#1}{#2}{#4}
\strand{#1+0.5*#3}{#2}{#4}
\capcup{#1+#3}{#2}{#3}{#4}
}

\newcommand{\tlcapcup}[4]{
\capcup{#1}{#2}{#3}{#4}
\strand{#1+#3}{#2}{#4}
}

\newcommand{\trcapcup}[4]{
\strand{#1}{#2}{#4}
\capcup{#1+0.5*#3}{#2}{#3}{#4}
}

\newcommand{\hookl}[4]{
\draw(#1,#2+#4) to [out=90,in=90] (#1-0.2*#3,#2+#4) -- (#1-0.2*#3,#2) to [out=270,in=270] (#1,#2);
}

\newcommand{\bslash}[4]{
\draw(#1+#3,#2) -- (#1,#2+#4);
}

\newcommand{\bslashcps}[4]{
\bslash{#1}{#2}{#3}{#4}
\jcap{#1}{#2}{#3}{#4}
\jcup{#1+0.5*#3}{#2}{#3}{#4}
}

\newcommand{\lbslash}[4]{
\bslashcps{#1}{#2}{#3}{#4}
\strand{#1+1.5*#3}{#2}{#4}
}

\newcommand{\rbslash}[4]{
\strand{#1}{#2}{#4}
\bslashcps{#1+0.5*#3}{#2}{#3}{#4}
}

\newcommand{\dbslash}[4]{
\bslash{#1}{#2}{#3}{#4}
\bslash{#1+0.5*#3}{#2}{#3}{#4}
\jcap{#1}{#2}{#3}{#4}
\jcup{#1+#3}{#2}{#3}{#4}
}

\newcommand{\jslash}[4]{
\draw(#1,#2) -- (#1+#3,#2+#4);
}

\newcommand{\slashcps}[4]{
\jslash{#1}{#2}{#3}{#4}
\jcup{#1}{#2}{#3}{#4}
\jcap{#1+0.5*#3}{#2}{#3}{#4}
}

\newcommand{\lslash}[4]{
\slashcps{#1}{#2}{#3}{#4}
\strand{#1+1.5*#3}{#2}{#4}
}

\newcommand{\rslash}[4]{
\strand{#1}{#2}{#4}
\slashcps{#1+0.5*#3}{#2}{#3}{#4}
}

\newcommand{\ducup}[4]{
\jcap{#1}{#2}{#3}{#4}
\jcap{#1+#3}{#2}{#3}{#4}
\jcup{#1+0.5*#3}{#2}{#3}{#4}
\jcup{#1}{#2-0.5*#4}{3*#3}{1.5*#4}
}

\newcommand{\dlcap}[4]{
\jcup{#1}{#2}{#3}{#4}
\jcup{#1+#3}{#2}{#3}{#4}
\jcap{#1+0.5*#3}{#2}{#3}{#4}
\jcap{#1}{#2}{3*#3}{1.5*#4}
}

\newcommand{\dulcp}[4]{
\jcap{#1+0.5*#3}{#2}{#3}{#4}
\jcup{#1+0.5*#3}{#2}{#3}{#4}
\jcap{#1}{#2}{3*#3}{1.5*#4}
\jcup{#1}{#2-0.5*#4}{3*#3}{1.5*#4}
}

\newtheorem{lemma}{Lemma}[section]

\newtheorem{theorem}[lemma]{Theorem}

\newtheorem{conjecture}[lemma]{Conjecture}

\theoremstyle{definition}
\newtheorem{remark}[lemma]{Remark}

\newtheorem{example}[lemma]{Example}

\begin{document}
\parindent0em
\setlength\parskip{.1cm}
\thispagestyle{empty}
\title{Arbitrarily large torsion in Khovanov cohomology}

\author{Sujoy Mukherjee}
\address{Department of Mathematics, The Ohio State University, Columbus OH, USA}
\email{mukherjee.166@osu.edu}

\author{Dirk Sch\"utz}
\address{Department of Mathematical Sciences, Durham University, Durham, UK}
\email{dirk.schuetz@durham.ac.uk}

\date{\today.}

\begin {abstract}
For any positive integer $k$ and $p\in \{3,5,7\}$ we construct a link which has a direct summand $\Z/p^k\Z$ in its Khovanov cohomology.
\end {abstract}

\maketitle

\section{Introduction}
Khovanov cohomology was introduced in \cite{MR1740682} as a categorification of the Jones polynomial \cite{MR908150} and has since been proven to be an invaluable tool in knot theory. Extensive calculations, see for example \cite{MR2034399, MR3205577}, show an abundance of $2$-torsion in it, while other torsion appears much more rarely. Indeed, less than $200$ of the prime knots with at most $16$ crossings have $4$-torsion in their Khovanov cohomology, and none have $3$-torsion or of order larger than $4$ \cite{MR3205577}.

In \cite{MR2320156} Bar-Natan introduced a more efficient algorithm to calculate Khovanov cohomology with which he detected torsion of order $3$ and $5$ for the torus knot $T(5,6)$, and torsion of order $7$ for $T(7,8)$. Based on this algorithm, in \cite{MR3894728} more examples of knots and links which admit $3$, $5$, and $7$-torsion are given. 
Furthermore, they exhibit the flat $2$-cabling of $T(2,2k+1)$, a $2$-component link of braid index $4$, as a potential example of a link admitting $2^k$-torsion in its Khovanov cohomology. Computationally this has been verified up to $k=23$. While the computed cohomology groups follow a certain pattern which makes it easy to believe this to be true for arbitrary $k$, obtaining a theoretical argument is not obvious.

More recently, in \cite{2019arXiv190606278M} the first author introduces links which admit torsion of order $9$, $27$, $81$, and $25$ in their Khovanov cohomology. Interestingly, these examples are based on connected sums. In contrast, Asaeda and Przytycki, in \cite{MR2147419}, have shown that taking a connected sum of a link with the Hopf link creates an additional copy of the torsion groups present in the Khovanov cohomology of the link but does not create larger ones.

Consider the link $L_3$ which is the closure of the braid word $(\sigma_1\sigma_2\sigma_3)^4 \sigma_1\sigma_2$. Notice that this link is one crossing short of the torus knot $T(4,5)$. Furthermore, it has two components, one of which is an unknot and the other the $T(3,4)$ torus knot; see Figure \ref{fig:thelink} for a diagram.
\begin{figure}[ht]
\includegraphics[height=3cm, width=5cm]{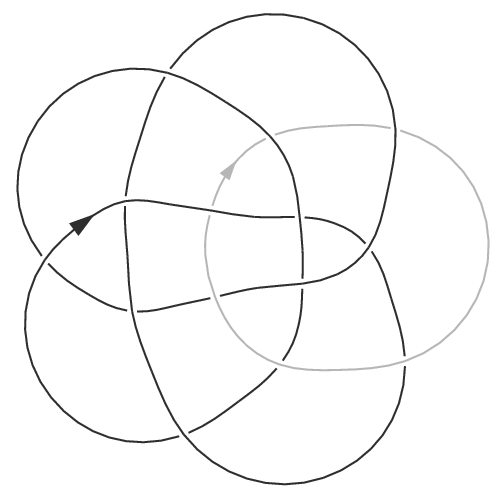}
\caption{\label{fig:thelink}The link $L_3$.}
\end{figure}

It was conjectured in \cite{2019arXiv190606278M} that the Khovanov cohomology of the connected sum 
\[
L^k_3 = L_3 \co \cdots \co L_3 \co T(2,3),
\]
where we assume to have $k$ factors of $L_3$ in the connected sum, contains a direct summand $\Z/3^l\Z$ for all $l\in \{1,\ldots,k\}$. Since $L_3$ is a link, we need to be more precise how the connected sums are formed. We declare that in any consecutive connected sum $L_3\co L_3$ we connect the unknot component of the left $L_3$ with the $T(3,4)$ component of the right $L_3$, and the last $L_3$ has its unknot component connected to $T(2,3)$.

\begin{figure}[ht]
\includegraphics[height=3cm, width=12cm]{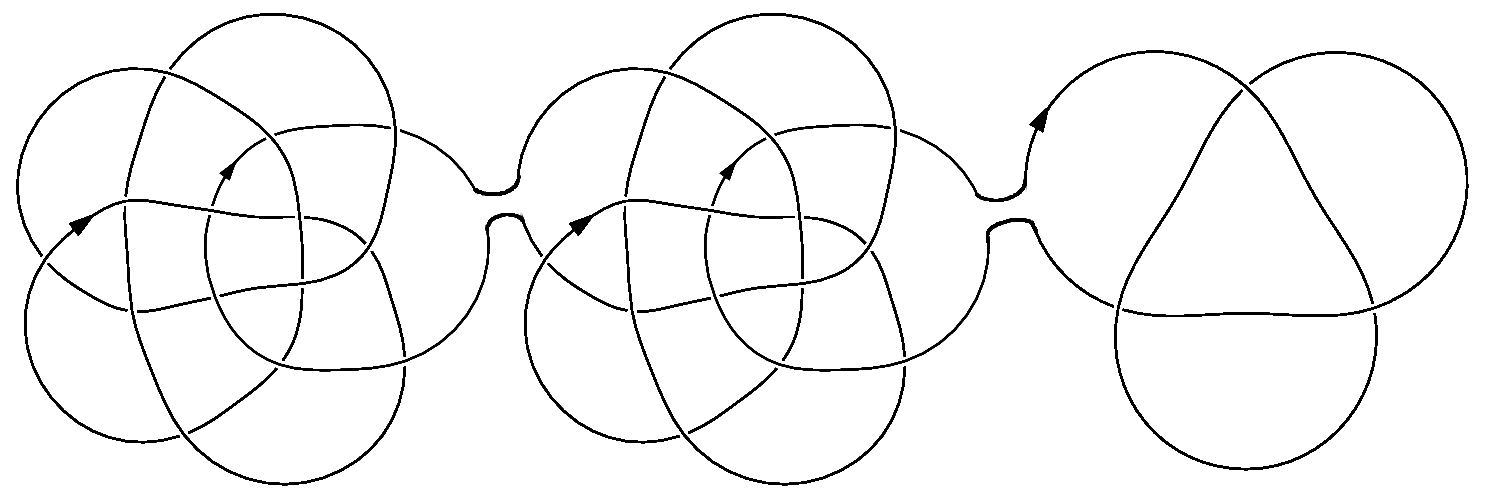}
\caption{The link $L_3^2$.}
\end{figure}

The above mentioned conjecture was based on computer calculations for $k\leq 4$. A striking feature of these calculations is that a $\Z/3^k\Z$ direct summand appears in the highest non-zero homological degree, and the second highest quantum degree. This turns out to be advantageous, since the Khovanov cochain complexes are more accessible towards the ends of the homological degrees.

Furthermore, the cochain complexes for a connected sum are algebraically related to the individual cochain complexes, suggesting these torsion summands should be theoretically justifiable. We show that this is indeed the case.

\begin{theorem}\label{thm:maintheorem}
Let $k$ be a positive integer. Then the Khovanov cohomology of the link $L^k_3$ contains direct summands $\Z/3^l\Z$ for all $l\in \{1,\ldots,k\}$.
\end{theorem}

One may ask whether this works for numbers different from $3$. For any positive integer $n$ we can define a two component link $L_n$ by taking the closure of the braid word $(\sigma_1 \cdots \sigma_n)^{n+1}\sigma_1 \cdots \sigma_{n-1}$, and from this we define a $(k+1)$ component link $L^k_n$ as above. The analogue of Theorem \ref{thm:maintheorem} does indeed hold for $n=5$ and $n=7$, so it may not be unreasonable to expect this result for any odd prime.

\begin{conjecture}\label{con:mainconj}
Let $p$ be an odd prime and $k$ a positive integer. Then the Khovanov cohomology of the link $L^k_p$ contains direct summands $\Z/p^l\Z$ for all $l\in \{1,\ldots,k\}$.
\end{conjecture}

Computations show that Conjecture \ref{con:mainconj} cannot work for $p=2$: the link $L^2_2$ does not have $4$-torsion in its Khovanov cohomology. We note however that our techniques make $2$ look special compared to odd primes. We are forced to invert $2$ in order to simplify the cochain complexes, which then allows us to isolate a good subcomplex which is responsible for the odd torsion summands.

The links $L^k_4$ appear to be a better bet to create torsion of order $2^k$, as calculations for low values of $k$ show. However, the direct summands do not appear in the same pattern as for $p=3,5,$ or $7$, and our techniques would need to be somewhat refined in order to justify $2^k$-torsion. Until then we note that the largest $2$-power torsion we are aware of is the $2^{23}$-torsion observed in \cite{MR3894728}.

One may also wonder about the significance of the trefoil factor in $L_n^k$. As we shall see it does play an important role, although it appears that it can be replaced by any knot different from the unknot.

\subsection*{Acknowledgements}
The authors are grateful to Mikhail Khovanov and J\'{o}zef H.\ Przytycki for their interesting comments and suggestions.

\section{Khovanov cohomology of a connected sum}
\label{sec:connectedsum}
In his fundamental paper \cite{MR1740682} Khovanov introduced a finitely generated free bigraded cochain complex $\CK^\ast(L)$ over $\Z$ for any link diagram $L$ such that the resulting cohomology groups are link invariants.

An important observation in \cite{MR2034399} is that this cochain complex can be considered a finitely generated free complex over $R=\Z[X]/\langle X^2\rangle$ by choosing a basepoint on the link diagram. 

If two based links $L_1$ and $L_2$ are given, we can form their connected sum $L_1\co L_2$ along the basepoints, and by \cite[Prop.3.3]{MR2034399} we can identify
\begin{equation}\label{eq:connectedsum}
 \CK^\ast(L_1\co L_2) \cong \CK^\ast(L_1) \otimes_R \CK^\ast(L_2).
\end{equation}

Notice that since $R$ is commutative, the tensor product also has the structure of an $R$-complex. This corresponds to putting the basepoint for $L_1\co L_2$ on an arc involved in the connected sum. In view of the connected sum we do for $L^k_3$, this is not what we want. 

To resolve this, we put a basepoint on each component of $L_3$, and consider $\CK^\ast(L_3)$ as an $R-R$ bimodule chain complex, with the left action coming from using the basepoint on the $T(3,4)$-component, and the right action coming from using the basepoint on the unknot component.

Since $R$ is commutative, we can think of an $R-R$ bimodule as an $R\otimes R$ left module. In particular, we treat $R\otimes R$ as a free $R-R$ bimodule. However, $R$ itself is not free as an $R\otimes R$ module.

Let us turn $R$ into a graded ring by placing $1\in R$ in grading $1$ and $X$ in grading $-1$. Denote
\[
 \mu \colon R\otimes R \to R\{1\}
\]
the usual multiplication map, and let
\[
 \Delta\colon R \to R\otimes R\{1\}
\]
be given by
\[
 \Delta(1) = 1\otimes X + X\otimes 1, \Delta(X) = X \otimes X.
\]
Here $\{1\}$ denotes a grading shift which makes these maps grading preserving.

\begin{lemma}\label{lem:tensorR}
With the notations as above,
\[
 (R\otimes R)\otimes_R (R\otimes R) \cong R\otimes R \{1\} \oplus R\otimes R\{-1\},
\]
as $R-R$ bimodules. Furthermore, the basis of $(R\otimes R)\otimes_R (R\otimes R)\cong R\otimes R\otimes R$, when viewed as a left $R\otimes R$ module, is given by $1\otimes 1\otimes 1$ and $1\otimes X\otimes 1$.
\end{lemma}

\begin{proof}
Obviously, $(R\otimes R)\otimes_R (R\otimes R) \cong R\otimes R\otimes R$. Also, the latter is generated by $1\otimes 1 \otimes 1$ and $1\otimes X \otimes 1$ as a $R-R$ bimodule. Furthermore, the $R-R$ bimodule map
\[
 \varphi\colon R\otimes R \{1\} \oplus R\otimes R\{-1\} \to R\otimes R\otimes R
\]
defined by sending $1\otimes 1\in R\otimes R \{1\}$ to $1\otimes 1 \otimes 1$, and $1\otimes 1\in R\otimes R\{-1\}$ to $1\otimes X\otimes 1$ is grading preserving, and easily seen to be an isomorphism.
\end{proof}

For every $n\in \Z$ define a $R-R$ bimodule cochain complex $C^\ast(n)$ concentrated in homological degrees $0$ and $1$ by $C^0(n)=R\otimes R\{-1\}$, $C^1(n) = R\otimes R\{1\}$, and the coboundary $\delta_n\colon C^0(n) \to C^1(n)$ given by
\[
 \delta_n(1\otimes 1) = nX\otimes 1 - 1\otimes X.
\]
We need a notation to indicate a shift in homological degrees, which we express by
\[
 C^\ast[k] = C^{\ast-k}
\]
for $k\in \Z$, and $C^{\ast}$ a general cochain complex.

\begin{lemma}\label{lem:tensorC}
Let $n,m\in \Z$. Then
\[
 C^\ast(n) \otimes_R C^\ast(m) \simeq C^\ast(-nm)\{-2\} \oplus C^\ast(nm)[1]\{2\}
\]
as $R-R$ bimodule complexes. Here $\simeq$ means chain homotopy equivalent.
\end{lemma}

\begin{proof}
We write $C^\ast(n) \otimes_R C^\ast(m)$ as
\[
\begin{tikzpicture}
\node at (0,1.5) {$R\otimes R\otimes R\{-2\}$};
\node at (4,1.5) {$R\otimes R\otimes R$};
\node at (4,0) {$R\otimes R\otimes R$};
\node at (8,0) {$R\otimes R\otimes R\{2\}.$};
\draw[->] (1.3,1.5) -- node [above] {$\delta_n\otimes 1$} (3.1,1.5);
\draw[->] (1.3,1.4) -- node [above,sloped] {$1\otimes \delta_m$} (3.1,0.1);
\draw[->] (4.9,1.4) -- node [above,sloped] {$-1\otimes \delta_m$} (6.8,0.1);
\draw[->] (4.9,0) -- node [above] {$\delta_n\otimes 1$} (6.8,0);
\end{tikzpicture}
\]
Using Lemma \ref{lem:tensorR} we can write this as
\[
\begin{tikzpicture}
\node at (0,3) {$R\otimes R\{-1\}$};
\node at (0,2) {$R\otimes R\{-3\}$};
\node at (5,3) {$R\otimes R\{1\}$};
\node at (5,2) {$R\otimes R\{-1\}$};
\node at (5,1) {$R\otimes R\{1\}$};
\node at (5,0) {$R\otimes R\{-1\}$};
\node at (10,1) {$R\otimes R\{3\}$};
\node at (10,0) {$R\otimes R\{1\}$};
\draw[->] (0.9,3.15) -- node [above,sloped,scale=0.6] {$nX\otimes 1$} (4.1,3);
\draw[->] (0.9,3.05) -- node [above,sloped,scale=0.6] {$1\otimes 1$} (4.1,2.1);
\draw[->] (0.9,2.95) -- node [above,sloped,scale=0.6,near end] {$-1\otimes X$} (4.1,1);
\draw[->] (0.9,2.85) -- node [above,sloped,scale=0.6,near end] {$m\otimes 1$} (4.1,0.1);
\draw[-,line width=6pt,draw=white] (0.9,2) -- (4,2);
\draw[->] (0.9,2) -- node [above,scale=0.6,near end] {$nX\otimes 1$} (4.1,2);
\draw[->] (0.9,1.9) -- node [above,sloped,scale=0.6] {$-1\otimes X$} (4.1,0);
\draw[->] (5.9,3.05) -- node [above,sloped,scale=0.6] {$1\otimes X$} (9.2,1.1);
\draw[->] (5.9,2.95) -- node [above,sloped,scale=0.6] {$-m\otimes 1$} (9.2,0.2);
\draw[->] (5.9,2) -- node [above,sloped,scale=0.6,near start] {$1\otimes X$} (9.2,0.1);
\draw[-,line width=6pt,draw=white] (5.9,1) -- (9,1);
\draw[->] (5.9,1) -- node [above,scale=0.6,near start] {$nX\otimes 1$} (9.2,1);
\draw[->] (5.9,0.9) -- node [above,scale=0.6,sloped] {$-1\otimes 1$} (9.2,0);
\draw[->] (5.9,0) -- node [above,scale=0.6,sloped] {$nX\otimes 1$} (9.2,-0.1);
\end{tikzpicture}
\]
Between the various direct summands we detect two isomorphisms, namely the identity between the $R\otimes R\{-1\}$ summands in homological degrees $0$ and $1$ (the higher one in homological degree $1$), and $-$identity between the $R\otimes R\{1\}$ summands in homological degrees $1$ and $2$ (the lower one in homological degree $1$).

We can now perform \em Gaussian elimination \em \cite[Lem.3.2]{MR2320156} on these direct summands to get
\[
\begin{tikzpicture}
\node at (0,1) {$R\otimes R\{-3\}$};
\node at (4,1) {$R\otimes R\{1\}$};
\node at (4,0) {$R\otimes R\{-1\}$};
\node at (8,0) {$R\otimes R\{3\}$};
\draw[->] (0.9,1) -- node [above,sloped,scale=0.6] {$-1\otimes X - nmX\otimes 1$} (3.1,0);
\draw[->] (4.8,1) -- node [above,sloped,scale=0.6] {$1\otimes X - nmX\otimes 1$} (7.2,0);
\end{tikzpicture}
\]
Notice that Gaussian elimination leads to a `zig-zag' for the surviving direct summands. In particular, the first coboundary has the $-1\otimes X$ summand from the previous complex, while the $-nmX\otimes 1$ summand is the `zig-zag' coming from the composition $-(m\otimes 1)\circ (1\otimes 1)^{-1}\circ (nX\otimes 1)$.

Also, the horizontal arrows are $0$ because $nX\cdot nX=0$ in $R$.

Since Gaussian elimination preserves the chain homotopy type over an additive category \cite{MR2320156}, we get the result.
\end{proof}

\begin{remark}
As an abelian group, $R\otimes R$ is a free abelian group of rank $4$, generated by $1\otimes 1$, $X\otimes 1$, $1\otimes X$ and $X\otimes X$. The matrix of $\delta_n$ in terms of this basis is given by
\[
 \Delta_n = 
\begin{pmatrix} 
0 & 0 & 0 & 0 \\
n & 0 & 0 & 0 \\
-1 & 0 & 0 & 0 \\
0 & -1 & n & 0  
\end{pmatrix}
.
\]
It follows that the cohomology of this cochain complex is free abelian of rank $2$ in both homological degrees.
\end{remark}
For $m\in \Z$ define a left $R$-module cochain complex $D^\ast(m)$ concentrated in homological degrees $0$ and $1$ by $D^0(m) = R\{-1\}$, $D^1(m)=R\{1\}$, with coboundary $\nu_m$ given by
\[
 \nu_m(1) = mX.
\]
Clearly the cohomology of this complex, treated as abelian groups, has torsion of order $m$, but more importantly we have the following.

\begin{lemma}\label{lem:tensorCD}
Let $n,m\in \Z$. Then
\[
 C^\ast(n) \otimes_R D^\ast(m) \simeq D^\ast(nm)\{-2\} \oplus D^\ast(nm)[1]\{2\}
\]
as left $R$-module complexes.
\end{lemma}

\begin{proof}
This is very similar to the proof of Lemma \ref{lem:tensorC}. We have $C^\ast(n) \otimes_R D^\ast(m)$ is given by
\[
\begin{tikzpicture}
\node at (0,1.5) {$R\otimes R\{-2\}$};
\node at (4,1.5) {$R\otimes R$};
\node at (4,0) {$R\otimes R$};
\node at (8,0) {$R\otimes R\{2\}.$};
\draw[->] (1,1.5) -- node [above] {$\delta_n\otimes 1$} (3.4,1.5);
\draw[->] (1,1.4) -- node [above,sloped] {$1\otimes mX$} (3.4,0);
\draw[->] (4.6,1.5) -- node [above,sloped] {$-1\otimes mX$} (7.15,0.1);
\draw[->] (4.6,0) -- node [above] {$\delta_n\otimes 1$} (7.15,0);
\end{tikzpicture}
\]
Since we treat this as a left $R$-module complex, we can use $R\otimes R\cong R\{1\}\oplus R\{-1\}$ as left $R$-modules by the same argument as in Lemma \ref{lem:tensorR}. The basis of $R\otimes R$ is given by $1\otimes 1$ and $1\otimes X$. In this basis the cochain complex is
\[
\begin{tikzpicture}
\node at (0,3) {$R\{-1\}$};
\node at (0,2) {$R\{-3\}$};
\node at (4,3) {$R\{1\}$};
\node at (4,2) {$R\{-1\}$};
\node at (4,1) {$R\{1\}$};
\node at (4,0) {$R\{-1\}$};
\node at (8,1) {$R\{3\}$};
\node at (8,0) {$R\{1\}$};
\draw[->] (0.6,3.1) -- node [above,sloped] {$nX$} (3.5,3);
\draw[->] (0.6,3.0) -- node [above,sloped] {$-1$} (3.4,2.1);
\draw[->] (0.6,2.9) -- node [above,sloped,near end] {$m$} (3.5,0);
\draw[-,line width=6pt,draw=white] (0.6,2) -- (3.2,2);
\draw[->] (0.6,2) -- node [above] {$nX$} (3.4,2);
\draw[->] (4.5,3) -- node [above,sloped] {$-m$} (7.5,0.2);
\draw[-,line width=6pt,draw=white] (4.5,1) -- (7.2,1);
\draw[->] (4.5,1) -- node [above,near start] {$nX$} (7.5,1);
\draw[->] (4.5,0.9) -- node [above,sloped] {$-1$} (7.5,0.1);
\draw[->] (4.6,0) -- node [above] {$nX$} (7.5,0);
\end{tikzpicture}
\]
Again we can perform two Gaussian eliminations, after which the surviving $R\{-3\}$ summand in homological degree $0$ together with the surviving $R\{-1\}$ summand in homological degree $1$ form the complex $D^\ast(nm)\{-2\}$. The remaining two summands form a direct summand $D^\ast(-nm)[1]\{2\}$, but we can remove the $-1$ factor with a change of basis.
\end{proof}

The complexes $D^\ast(m)$ can also be viewed as right $R$-module complexes or $R-R$ bimodule complexes. The reader may want to convince themselves that the analogous statement of Lemma \ref{lem:tensorCD} for the right $R$-module complex $D^\ast(m)\otimes_R C^\ast(n)$ cannot be derived in this way.
In fact, this is not possible, as the Khovanov cohomology of $T(2,3)\co L_3$ does not contain $3$-torsion when we connect $T(2,3)$ to the $T(3,4)$-component of $L_3$.

Let us introduce another $R-R$ bimodule cochain complex $E^\ast$, concentrated in homological degrees $0$ and $1$, as follows. We set $E^0 = R\otimes R$, $E^1 = R\{1\}$ and the coboundary is given by the multiplication map $\mu$, which is a bimodule map.

\begin{lemma}\label{lem:removeE}
Let $n\in \Z$. The $R-R$ bimodule complex $C^\ast(m)\otimes_R E^\ast$ is chain homotopy equivalent to $C^\ast(-m)\{-1\}$ as an $R-R$ bimodule complex.
\end{lemma}

\begin{proof}
Using Lemma \ref{lem:tensorR} we can write $C^\ast(n)\otimes_R E^\ast$ as
\[
\begin{tikzpicture}
\node at (0,2.5) {$R\otimes R$};
\node at (0,1.5) {$R\otimes R\{-2\}$};
\node at (4,2.5) {$R\otimes R \{2\}$};
\node at (4,1.5) {$R\otimes R$};
\node at (4,0) {$R\otimes R$};
\node at (8,0) {$R\otimes R\{2\}$};
\draw[->] (0.7,2.5) -- node [above,scale=0.7] {$nX\otimes 1$} (3.2,2.5);
\draw[->] (0.7,2.4) -- node [above,sloped,scale=0.7] {$-1\otimes 1$} (3.4,1.6);
\draw[->] (0.7,2.3) -- node [above,sloped,near end,scale=0.7] {$1\otimes 1$} (3.4,0.1);
\draw[-,line width=6pt,draw=white] (0.8,1.5) -- (3,1.5);
\draw[->] (1,1.5) -- node [above,scale=0.7] {$nX\otimes 1$} (3.4,1.5);
\draw[->] (1,1.4) -- node [below,sloped,scale=0.7] {$1\otimes X$} (3.4,0);
\draw[->] (4.8,2.5) -- node [above,sloped,scale=0.7] {$-1\otimes 1$} (7.2,0.2);
\draw[->] (4.6,1.5) -- node [above,sloped,scale=0.7] {$-1\otimes X$} (7.2,0.1);
\draw[->] (4.6,0) -- node [above,scale=0.7] {$\delta_n$} (7.2,0);
\end{tikzpicture}
\]
We can use Gaussian elimination on the morphism $-1\otimes 1$ between the $R\otimes R\{2\}$, and after that, we use Gaussian elimination on the $1\otimes 1$ morphism between the $R\otimes R$ direct summands in homological degrees $0$ and $1$.

This leads to
\[
\begin{tikzpicture}
\node at (0,0) {$R\otimes R\{-2\}$};
\node at (4,0) {$R\otimes R$};
\draw[->] (1,0) -- node [above,scale=0.7] {$nX\otimes 1+1\otimes X$} (3.4,0);
\end{tikzpicture}
\]
which implies the statement.
\end{proof}

In view of Lemma \ref{lem:tensorC} and Lemma \ref{lem:tensorCD} we would like to find a knot $K$ with $D^\ast(m)$ as a direct summand in the chain homotopy type of its Khovanov complex, and a two component link $L$ with $C^\ast(n)$ as a direct summand in the chain homotopy type of its Khovanov complex. As we shall see, the trefoil knot works with $m=2$, but to get an appropriate $L$ we need to simplify the algebra.

\section{A recap of Bar-Natan's algorithm}
In \cite{MR2174270} Bar-Natan gave a new introduction to Khovanov cohomology based on tangles and cobordisms. Furthermore, in \cite{MR2320156} he used this to obtain a fast algorithm to calculate it. We quickly recall his construction, and show how it can be used to keep the information coming from the action of $R$.

Given a finite subset $B\subset S^1$, let $\Cobd(B)$ be the category whose objects are smooth compact submanifolds $S\subset D^2$ with $\partial S = B$, and whose morphisms are ``dotted'' cobordisms embedded in a cylinder $D^2\times [0,1]$, up to boundary preserving isotopy. Here ``dotted'' means that we allow finitely many points in the interior of a cobordism, which are allowed to move freely.

Now define $\Cob(B)$ to be the pre-additive category with the same objects as $\Cobd(B)$, and where the morphism groups are obtained by taking the free abelian group of the morphisms from $\Cobd(B)$, and adding the local relations
\[
\begin{tikzpicture}[baseline={([yshift=-.5ex]current bounding box.center)}]
\sphere{0}{0}{0.5}
\node at (1,0) {$=0$,};
\spheredot{3}{0}{0.5}
\node at (4,0) {$=1$,};
\plane{5}{-0.5}
\node at (5.5,0.2) [scale = 0.75] {$\bullet$};
\node at (5.5,-0.2) [scale = 0.75] {$\bullet$};
\node at (6.5,0) {$=0$,};
\end{tikzpicture}
\]
and
\[
\begin{tikzpicture}[baseline={([yshift=-.5ex]current bounding box.center)}]
\cylinder{2.2}{0}{1}{2}
\node at (3.7,1) {$=$};
\bowldot{4.4}{2}{1}
\bowlud{4.4}{0}{1}
\node at (5.9,1) {$+$};
\bowl{6.6}{2}{1}
\bowluddot{6.6}{0}{1}
\node at (8,0.8) {.};
\end{tikzpicture}
\]
We can turn this category into an additive category by formally adding direct sums as in \cite{MR2174270}. We then let $\mathfrak{K}(\Cob(B))$ be the category of cochain complexes over the additive category.

Given a tangle $T$, \cite{MR2174270} then assigns a cochain complex $\CK^\ast(T)$ as an object in $\mathfrak{K}(\Cob(\partial T))$. The algorithm to calculate Khovanov cohomology described in \cite{MR2320156} can now be summarized as follows. We refer to the original publication for more details.
\begin{enumerate}
\item Consider the tangle $T$ as a sequence of tangles $T_1,\ldots,T_k$, with each subtangle $T_i$ consisting of one crossing. Form $\CK^\ast(T_1)$.
\item (Tensor product) Assuming we have a chain complex $C^\ast$ representing the chain homotopy type of the tangle $T_1\cdots T_{i-1}$ for some $i\geq 2$, form the tensor product $C^\ast \otimes \CK^\ast(T_i)$. To get the new objects in this tensor product, we need to combine the boundaries of the $1$-dimensional manifolds according to the gluings from the tangles.
\item (Delooping \cite[Lem.3.1]{MR2320156}) In this new cochain complex some of the generators will have circle components in the representing $1$-manifold. Such a generator can be replaced by two generators without the circle. Repeat until all circles are gone.
\item (Gaussian elimination \cite[Lem.3.2]{MR2320156}) The delooped complex may have several direct summands, on which Gaussian elimination can be performed. This is repeated until no further eliminations are possible. The resulting cochain complex $C^\ast$ has the chain homotopy type of $\CK^\ast(T_1\cdots T_i)$.
\item Continue steps (2)-(4) until the last tangle, after which we have a cochain complex $C^\ast$ chain homotopy equivalent to $\CK^\ast(T)$.
\end{enumerate}
If we start with a link diagram $L$, the final result is a cochain complex over $\Cob(\emptyset)$, and all generators have the empty set as their object. The cobordisms can be reduced to the empty set using the relations, and the information boils down to a cochain complex over $\Z$.

In view of Section \ref{sec:connectedsum} we would like to get a cochain complex over $R$ or $R\otimes R$. Now if we choose a basepoint on a tangle $T$, we can get a cochain map $X^\ast\colon \CK^\ast(T)\to \CK^\ast(T)$ by putting a dot on the component of the cylinder corresponding to the basepoint.

For the algorithm, we only need one of the tangles to have the basepoint, and this will give a cochain map $X^\ast$ on the final cochain complex $C^\ast$. In the case of a link diagram, this turns $C^\ast$ into a left $R$-module complex. Similarly, with two basepoints we can get $C^\ast$ to be a $R-R$ bimodule complex. 

\begin{example}
(Compare \cite[\S 6.2]{MR1740682}) Consider the trefoil knot $T(2,3)$ obtained from the braid word $\sigma_1^3$. The cochain complex $\CK^\ast(\sigma_1\sigma_1) = \CK^\ast(\sigma_1)\otimes \CK^\ast(\sigma_1)$ is given by
\[
\begin{tikzpicture}
\draw (0,2) to [out=90, in=270] (0.1,2.25) to [out=90, in=270] (0,2.5) to [out=90, in=270] (0.1,2.75) to [out=90, in=270] (0,3);
\draw (0.6,2) to [out=90, in=270] (0.5,2.25) to [out=90,in=270] (0.6,2.5) to [out=90,in=270] (0.5,2.75) to [out=90, in=270] (0.6,3);
\draw (3,3) to [out=90, in = 180] (3.3,3.15) to [out=0, in=90] (3.6,3);
\draw (3,4) to [out=270, in=90] (3.1,3.75) to [out=270,in=90] (3,3.5) to [out=270,in=180] (3.3,3.35) to [out=0, in =270] (3.6,3.5) to [out=90, in=270] (3.5,3.75) to [out=90, in=270] (3.6,4);
\draw (3,1) to [out=90, in=270] (3.1,1.25) to [out=90, in=270] (3,1.5) to [out=90, in=180] (3.3,1.65) to [out=0, in=90] (3.6,1.5) to [out=270,in=90] (3.5,1.25) to [out=270,in=90] (3.6,1);
\draw (3,2) to [out=270,in=180] (3.3,1.85) to [out=0, in=270] (3.6,2);
\draw (6,2) to [out=90,in=180] (6.3,2.15) to [out=0, in=90] (6.6,2);
\draw (6,2.5) to [out=90,in=180] (6.3,2.65) to [out=0, in=90] (6.6,2.5);
\draw (6,2.5) to [out=270,in=180] (6.3,2.35) to [out=0,in=270] (6.6,2.5);
\draw (6,3) to [out=270,in=180] (6.3,2.85) to [out=0,in=270] (6.6,3);
\draw[->] (1,2.6) -- node [above,sloped] {$S$} (2.6,3.5);
\draw[->] (1,2.4) -- node [above,sloped] {$S$} (2.6,1.5);
\draw[->] (4,3.5) -- node [above,sloped] {$S$} (5.6,2.6);
\draw[->] (4,1.5) -- node [above,sloped] {$-S$} (5.6,2.4);
\node at (4,3) {$\{1\}$};
\node at (4,1) {$\{1\}$};
\node at (7,2) {$\{2\}$};
\end{tikzpicture}
\]
where $S$ denotes the obvious surgery cobordism. After delooping this turns into
\[
\begin{tikzpicture}
\draw (0,2) to [out=90, in=270] (0.1,2.25) to [out=90, in=270] (0,2.5);
\draw (0.6,2) to [out=90, in=270] (0.5,2.25) to [out=90,in=270] (0.6,2.5);
\draw (3,3) to [out=90, in=180] (3.3,3.15) to [out=0, in=90] (3.6,3);
\draw (3,3.5) to [out=270,in=180] (3.3,3.35) to [out=0, in=270] (3.6,3.5);
\draw (3,1) to [out=90, in=180] (3.3,1.15) to [out=0, in=90] (3.6,1);
\draw (3,1.5) to [out=270,in=180] (3.3,1.35) to [out=0, in=270] (3.6,1.5);
\draw (6,3) to [out=90, in=180] (6.3,3.15) to [out=0, in=90] (6.6,3);
\draw (6,3.5) to [out=270,in=180] (6.3,3.35) to [out=0, in=270] (6.6,3.5);
\draw (6,1) to [out=90, in=180] (6.3,1.15) to [out=0, in=90] (6.6,1);
\draw (6,1.5) to [out=270,in=180] (6.3,1.35) to [out=0, in=270] (6.6,1.5);
\draw[->] (1,2.35) -- node [above,sloped] {$S$} (2.6,3.25);
\draw[->] (1,2.15) -- node [above,sloped] {$S$} (2.6,1.25);
\draw[->] (4,3.25) -- (5.6,3.25);
\draw[->] (4.4,2.8) -- node [above,sloped,near end] {$1$} (5.6,1.35);
\draw[-,line width=6pt,draw=white] (4,1.35) -- (5.6,3.15);
\draw[->] (4,1.35) -- (5.6,3.15);
\draw[->] (4,1.25) -- node [above] {$-1$} (5.6,1.25);
\node at (4,3) {$\{1\}$};
\node at (4,1) {$\{1\}$};
\node at (7,3) {$\{3\}$};
\node at (7,1) {$\{1\}$};
\draw (4.7,3.35) to [out=90,in=180] (4.8,3.4) to [out=0,in=90] (4.9,3.35);
\draw (4.7,3.55) to [out=270,in=180] (4.8,3.5) to [out=0,in=270] (4.9,3.55);
\node[scale=0.6] at (4.8,3.5) {$\bullet$};
\draw (4.3,2.05) to [out=90,in=180] (4.4,2.1) to [out=0,in=90] (4.5,2.05);
\draw (4.3,2.25) to [out=270,in=180] (4.4,2.2) to [out=0,in=270] (4.5,2.25);
\node[scale=0.6] at (4.4,2.1) {$\bullet$};
\node at (4.1,2.15) {$-$};
\end{tikzpicture}
\]
We can perform one Gaussian elimination to get
\[
\begin{tikzpicture}
\draw (0,2) to [out=90, in=270] (0.1,2.25) to [out=90, in=270] (0,2.5);
\draw (0.6,2) to [out=90, in=270] (0.5,2.25) to [out=90,in=270] (0.6,2.5);
\draw (3,2) to [out=90, in=180] (3.3,2.15) to [out=0, in=90] (3.6,2);
\draw (3,2.5) to [out=270,in=180] (3.3,2.35) to [out=0, in=270] (3.6,2.5);
\draw (6,2) to [out=90, in=180] (6.3,2.15) to [out=0, in=90] (6.6,2);
\draw (6,2.5) to [out=270,in=180] (6.3,2.35) to [out=0, in=270] (6.6,2.5);
\draw[->] (1,2.25) -- node [above] {$S$} (2.6,2.25);
\draw[->] (4,2.25) -- node [above] {$-$} (5.6,2.25);
\draw (4.4,2.4) to [out=90,in=180] (4.5,2.45) to [out=0,in=90] (4.6,2.4);
\draw (4.4,2.6) to [out=270,in=180] (4.5,2.55) to [out=0,in=270] (4.6,2.6);
\draw (5,2.4) to [out=90,in=180] (5.1,2.45) to [out=0,in=90] (5.2,2.4);
\draw (5,2.6) to [out=270,in=180] (5.1,2.55) to [out=0,in=270] (5.2,2.6);
\node[scale=0.6] at (4.5,2.55) {$\bullet$};
\node[scale=0.6] at (5.1,2.45) {$\bullet$};
\node at (4,2) {$\{1\}$};
\node at (7,2) {$\{3\}$};
\end{tikzpicture}
\]
Applying the algorithm to the next crossing, without closing the braid yet, is easily seen to lead to
\[
\begin{tikzpicture}
\draw (0,2) to [out=90, in=270] (0.1,2.25) to [out=90, in=270] (0,2.5);
\draw (0.6,2) to [out=90, in=270] (0.5,2.25) to [out=90,in=270] (0.6,2.5);
\draw (3,2) to [out=90, in=180] (3.3,2.15) to [out=0, in=90] (3.6,2);
\draw (3,2.5) to [out=270,in=180] (3.3,2.35) to [out=0, in=270] (3.6,2.5);
\draw (6,2) to [out=90, in=180] (6.3,2.15) to [out=0, in=90] (6.6,2);
\draw (6,2.5) to [out=270,in=180] (6.3,2.35) to [out=0, in=270] (6.6,2.5);
\draw (9,2) to [out=90, in=180] (9.3,2.15) to [out=0, in=90] (9.6,2);
\draw (9,2.5) to [out=270,in=180] (9.3,2.35) to [out=0, in=270] (9.6,2.5);
\draw[->] (1,2.25) -- node [above] {$S$} (2.6,2.25);
\draw[->] (4,2.25) -- node [above] {$-$} (5.6,2.25);
\draw[->] (7,2.25) -- node [above] {$+$} (8.6,2.25);
\draw (4.4,2.4) to [out=90,in=180] (4.5,2.45) to [out=0,in=90] (4.6,2.4);
\draw (4.4,2.6) to [out=270,in=180] (4.5,2.55) to [out=0,in=270] (4.6,2.6);
\draw (5,2.4) to [out=90,in=180] (5.1,2.45) to [out=0,in=90] (5.2,2.4);
\draw (5,2.6) to [out=270,in=180] (5.1,2.55) to [out=0,in=270] (5.2,2.6);
\node[scale=0.6] at (4.5,2.55) {$\bullet$};
\node[scale=0.6] at (5.1,2.45) {$\bullet$};
\draw (7.4,2.4) to [out=90,in=180] (7.5,2.45) to [out=0,in=90] (7.6,2.4);
\draw (7.4,2.6) to [out=270,in=180] (7.5,2.55) to [out=0,in=270] (7.6,2.6);
\draw (8,2.4) to [out=90,in=180] (8.1,2.45) to [out=0,in=90] (8.2,2.4);
\draw (8,2.6) to [out=270,in=180] (8.1,2.55) to [out=0,in=270] (8.2,2.6);
\node[scale=0.6] at (7.5,2.55) {$\bullet$};
\node[scale=0.6] at (8.1,2.45) {$\bullet$};
\node at (4,2) {$\{1\}$};
\node at (7,2) {$\{3\}$};
\node at (10,2) {$\{5\}$};
\end{tikzpicture}
\]
If we now close the strands of the braid, and consider each circle to give rise to a factor $R$, the resulting cochain complex is
\[
\begin{tikzpicture}
\node at (0,0) {$R\otimes R$};
\node at (2.5,0) {$R\{1\}$};
\node at (5,0) {$R\{3\}$};
\node at (7.5,0) {$R\{5\}$};
\draw[->] (0.7,0) -- node [above] {$\mu$} (2,0);
\draw[->] (3,0) -- node [above] {$0$} (4.5,0);
\draw[->] (5.5,0) -- node [above] {$2X$} (7,0);
\end{tikzpicture}
\]
If we treat the left strand as the based strand, the $0$-th cochain group $R\otimes R$ has the left $R$-module structure involving the first factor of $R$, which is isomorphic to $R\{-1\}\oplus R\{1\}$ as a left $R$-module. In particular, we can perform one more Gaussian elimination to get
\begin{equation}\label{eq:T23complex}
 \CK^\ast(T(2,3)) \simeq R^\ast\{3\} \oplus D^\ast(2)[2]\{7\}
\end{equation}
as left $R$-complexes. Here $R^\ast$ is the trivial left $R$-complex concentrated in homological degree $0$ given by $R^0 = R\{-1\}$. The shift by $+3$ in the quantum grading is coming from the three positive crossings in $T(2,3)$.
\end{example}

We notice that we can wait until the end before we commit to the basepoint. This also works if we have two basepoints that we can put on the last tangle. Another advantage of this is that at the last step, $B$ only has four points. After delooping, there are only two objects, and as morphisms (up to the local relations) we only have a surgery $S$, a surgery together with a dotting, which we denote by $\dot{S}$, and the various dottings on a cylinder (including no dottings at all).

\section{The Khovanov cochain complex for $L_n$}
We would like to get that $\CK^\ast(L_3)$ contains, up to chain homotopy, a direct summand complex $C^\ast(3)$, suitably shifted. This does not seem to be quite the case. However, it turns out that after inverting $2$ the Khovanov cochain complex becomes more amenable and we do get our direct summand.

Let $S$ be a subring of $\Q$ containing $1$. For any link diagram $L$ we write $\CK^\ast(L;S) = \CK^\ast(L)\otimes S$. We can then consider $\CK^\ast(L_3;S)$ as an $R_S - R_S$ bimodule complex, where $R_S = S[X]/\langle X^2\rangle$. We are mainly interested in $S=\Z_{(p)}$, the integers localized at a prime $p$, in which case we simply write $R_p = R_{\Z_{(p)}}$.

We also use the notations $C^\ast(n;R_S) = C^\ast(n)\otimes S$ and $D^\ast(m;R_S) = D^\ast(m)\otimes S$.

\begin{lemma}\label{lem:mainlemma}
There exists a finitely generated $\R - \R$ bimodule complex $F_3^\ast$ concentrated in homological degrees $0$ to $7$ such that
\[
 \CK^\ast(L_3;\Z_{(3)}) \simeq F_3^\ast \oplus C^\ast(3;\R)[8]\{25\}
\]
as $\R - \R$ bimodule complexes.
\end{lemma}

\begin{proof}
We apply Bar-Natan's algorithm on $T(4,5)$ by scanning the crossings according to the braid word $(\sigma_1\sigma_2\sigma_3)^5$. After $14$ crossings we get a cochain complex $C^\ast$ in $\mathfrak{K}(\Cob(B))$, where $B$ consists of four points, which is concentrated in homological degrees $0$ to $9$ and which ends in
\begin{equation}\label{eq:l3ending}
\begin{tikzpicture}[baseline={([yshift=-.5ex]current bounding box.center)}]
\node at (1,0) {$\{8\}$};
\node at (1,1.5) {$\{9\}$};
\node at (5,0) {$\{10\}$};
\node at (5,1.5) {$\{11\}$};
\node at (5,3) {$\{9\}$};
\node at (9,0) {$\{12\}$};
\node at (9,1.5) {$\{11\}$};
\smoothingud{0.1}{-0.25}{0.5}
\smoothinglr{0.1}{1.25}{0.5}
\smoothingud{4}{-0.25}{0.5}
\smoothinglr{4}{1.25}{0.5}
\smoothinglr{4.1}{2.75}{0.5}
\smoothingud{8}{-0.25}{0.5}
\smoothinglr{8}{1.25}{0.5}
\draw[->] (1.4,0.2) -- node [above,sloped,very near start,scale=0.8] {$-S$} (3.8,2.9);
\draw[->] (1.4,1.6) -- node [above,sloped,scale=0.8] {$2$} (3.8,3);
\draw[-,line width=6pt,draw=white] (1.4,1.5) -- (3.8,1.5);
\draw[-,line width = 12pt, draw=white] (2.4,1.8) -- (3.4,1.8);
\draw[-,line width=6pt,draw=white] (1.4,1.4) -- (3.8,0.1);
\draw[->] (1.4,1.4) -- node [above,sloped,near end,scale=0.8] {$-2S$} (3.8,0.1);
\draw[-,line width=6pt,draw=white] (1.4,0.1) -- (3.8,1.4);
\draw[->] (1.4,0.1) -- node [above,sloped,near end,scale=0.8] {$-\dot{S}$} (3.8,1.4);
\draw[->] (1.4,1.5) -- node [above] {$+$} (3.8,1.5);
\draw[->] (1.4,0) -- node [above] {$+$} (3.8,0);
\smoothingud{2.1}{0.1}{0.3}
\smoothingud{2.8}{0.1}{0.3}
\smoothinglr{2.1}{1.6}{0.3}
\smoothinglr{2.8}{1.6}{0.3}
\node[scale=0.8] at (2.15,0.25) {$\bullet$};
\node[scale=0.8] at (3.05,0.25) {$\bullet$};
\node[scale=0.8] at (2.25,1.825) {$\bullet$};
\node[scale=0.8] at (2.95,1.65) {$\bullet$};
\node[scale=0.8] at (6.56,0.26) {$\bullet$};
\node[scale=0.8] at (6.35,0.25) {$2$};
\draw[->] (5.4,3) -- node [above,sloped,scale=0.8,near start] {$\dot{S}$} (7.8,0.2);
\draw[-,line width=6pt,draw=white] (5.5,1.5) -- (7.8,1.5);
\draw[-,line width=6pt,draw=white] (5.5,0.1) -- (7.8,1.4);
\draw[->] (5.5,0.1) -- node [above,sloped,near start,scale=0.8] {$S$} (7.8,1.4);
\draw[->] (5.5,0) -- (7.8,0);
\draw[-,line width=6pt,draw=white] (6.05,1.075) -- (7.225,0.425);
\draw[->] (5.5,1.5) -- node [above,scale=0.8,near start] {$2$} (7.8,1.5);
\draw[->] (5.5,1.4) -- node [above,sloped,near start,scale=0.8] {$S$} (7.8,0.1);
\smoothingud{6.5}{0.1}{0.3}
\end{tikzpicture}
\end{equation}
We note that this was obtained with the assistance of a computer, and in Appendix \ref{app:handson} we show several stages in the algorithm.

We can treat this as the Khovanov complex for $L_3$ by connecting the two endpoints on the left, and connecting the two endpoints on the right. The resulting $R-R$ bimodule complex chain homotopy equivalent to $\CK^\ast(L_3)$ ends then in
\[
\begin{tikzpicture}
\node at (0,0) {$R\otimes R\{22\}$};
\node at (0,1.5) {$R\{23\}$};
\node at (5,0) {$R\otimes R\{24\}$};
\node at (5,1.5) {$R\{25\}$};
\node at (5,3) {$R\{23\}$};
\node at (10,0) {$R\otimes R\{26\}$};
\node at (10,1.5) {$R\{25\}$};
\draw[->] (1,0.2) -- node [above,sloped,near end,scale=0.8] {$-\mu$} (4.4,2.9);
\draw[->] (0.6,1.6) -- node [above,sloped,scale=0.8] {$2$} (4.4,3);
\draw[-,line width=6pt,draw=white] (0.6,1.5) -- (4.4,1.5);
\draw[-,line width=6pt,draw=white] (0.6,1.4) -- (4,0.1);
\draw[->] (0.6,1.4) -- node [above,sloped,scale=0.8,near end] {$-2\Delta$} (4,0.1);
\draw[-,line width=6pt,draw=white] (1,0.1) -- (4.4,1.4);
\draw[->] (1,0.1) -- node [above,sloped,scale=0.8,near end] {$-X\mu$} (4.4,1.4);
\draw[->] (0.6,1.5) -- node [above,scale=0.8,near end] {$2X$} (4.4,1.5);
\draw[->] (1,0) -- node [above,scale=0.8] {$\Delta\circ\mu$} (4,0);
\draw[->] (5.6,3) -- node [above,sloped,scale=0.8, near start] {$\Delta X$} (9,0.2);
\draw[-,line width=6pt,draw=white] (5.6,1.5) -- (9.4,1.5);
\draw[-,line width=6pt,draw=white] (6,0.1) -- (9.4,1.4);
\draw[->] (6,0.1) -- node [above,sloped,scale=0.8,near start] {$\mu$} (9.4,1.4);
\draw[-,line width=6pt,draw=white] (6.7,1.0) -- (7.9,0.525);
\draw[->] (5.6,1.4) -- node [above,sloped,scale=0.8,near start] {$\Delta$} (9,0.1);
\draw[->] (5.6,1.5) -- node [above,scale=0.8,near start] {$2$} (9.4,1.5);
\draw[->] (6,0) -- node [above,scale=0.8] {$2X\otimes 1$} (9,0);
\end{tikzpicture}
\]
The shift in the quantum grading comes from the $14$ positive crossings in $L_3$.

So far we have worked over the integers, but it is not clear whether we can improve this cochain complex significantly as an $R-R$ bimodule complex. But if we allow ourselves to invert $2$, there are two Gaussian eliminations that we can perform.

So we now switch to $\R$ and cancel the $\R\{23\}$ direct summands in homological degrees $7$ and $8$. The resulting morphism starting in $\R\otimes \R\{22\}$ and ending in $\R\otimes \R\{24\}$ is given by
\[
 \Delta\circ \mu - (-2\Delta)\circ \frac{1}{2} \circ (-\mu) = 0,
\]
and the morphism starting in $\R\otimes \R\{22\}$ and ending in $\R\{25\}$ is given by
\[
 -X\mu - (2X) \circ \frac{1}{2} \circ (-\mu) = 0.
\]
In particular, homological degrees $8$ and $9$ now form a direct summand cochain subcomplex. We can perform one more Gaussian elimination, after which the morphism between the remaining summands is given by
\begin{align*}
2X\otimes 1 & - \Delta \circ \frac{1}{2} \circ \mu = 2X\otimes 1 - \frac{1}{2}(X\otimes 1+1\otimes X) \\
& = \frac{1}{2} (3X\otimes 1 - 1 \otimes X)  = \frac{1}{2} \delta_3.
\end{align*}
As $\frac{1}{2}$ is a unit in $\R$ the result follows by a change of basis.
\end{proof}

For the next result we can work over $R$ again. Recall the $R-R$ bimodule complex $E^\ast$ from Section \ref{sec:connectedsum}.

\begin{lemma}\label{lem:beginner}
Let $n\geq 2$. There exists a finitely generated $R-R$ bimodule complex $G^\ast$ concentrated in homological degrees $2$ to $n(n+2)-1$ such that
\[
 \CK^\ast(L_n) \simeq E^\ast \{n(n+1)\} \oplus G^\ast
\]
as $R-R$ bimodule complexes.
\end{lemma}

\begin{proof}
We will only give the proof for $n=3$, the general case is similar. We apply Bar-Natan's scanning algorithm on the generating braid word, but only keep track of homological degrees $0$ and $1$ after each step.

For the first three steps, there are no possibilities to deloop or to use Gaussian elimination. The resulting cochain complex $\CK^\ast(\sigma_1\sigma_2\sigma_3)$ therefore begins with
\begin{equation}\label{eq:genericstart}
\begin{tikzpicture}[baseline={([yshift=-.5ex]current bounding box.center)}]
\fstrand{0}{0}{0.6}{0.6}
\lcapcup{3}{1.4}{0.6}{0.6}
\mcapcup{3}{0}{0.6}{0.6}
\rcapcup{3}{-1.4}{0.6}{0.6}
\draw[->] (1.1,0.4) -- node[above,sloped,scale=0.8] {$S$} (2.7,1.7);
\draw[->] (1.1,0.3) -- node[above,scale=0.8] {$S$} (2.7,0.3);
\draw[->] (1.1,0.2) -- node[above,sloped,scale=0.8] {$S$} (2.7,-1.1);
\draw[->] (4.1,1.7) -- (5,1.7);
\draw[->] (4.1,0.3) -- (5,0.3);
\draw[->] (4.1,-1.1) -- (5,-1.1);
\node at (5.5,0.3) {$\cdots$};
\node at (4.3,0) {$\{1\}$};
\node at (4.3,1.4) {$\{1\}$};
\node at (4.3,-1.4) {$\{1\}$};
\end{tikzpicture}
\end{equation}
We claim that homological degrees $0$ and $1$ remain in that form until we get to the cochain complex for the braid word $(\sigma_1\sigma_2\sigma_3)^4$. This is done by induction. Assume that the cochain complex for a subword of $(\sigma_1\sigma_2\sigma_3)^4$ begins as in (\ref{eq:genericstart}), and we tensor it with $\CK^\ast(\sigma_i)$ for $i\in \{1,2,3\}$. We will assume $i=2$, but the other cases are similar.

We then get
\[
\begin{tikzpicture}
\fstrand{0}{0}{0.6}{0.6}
\lcapcup{3}{1.4}{0.6}{0.3}
\fstrand{3}{1.7}{0.6}{0.3}
\mcapcup{3}{0}{0.6}{0.3}
\fstrand{3}{0.3}{0.6}{0.3}
\rcapcup{3}{-1.4}{0.6}{0.3}
\fstrand{3}{-1.1}{0.6}{0.3}
\fstrand{3}{-2.8}{0.6}{0.3}
\mcapcup{3}{-2.5}{0.6}{0.3}
\lcapcup{7}{0}{0.6}{0.3}
\mcapcup{7}{0.3}{0.6}{0.3}
\mcapcup{7}{-1.4}{0.6}{0.3}
\mcapcup{7}{-1.1}{0.6}{0.3}
\rcapcup{7}{-2.8}{0.6}{0.3}
\mcapcup{7}{-2.5}{0.6}{0.3}
\draw[->] (1.1,0.4) -- node [above,sloped,scale=0.8] {$S$} (2.7,1.7);
\draw[->] (1.1,0.3) -- node [above,scale=0.8] {$S$} (2.7,0.3);
\draw[->] (1.1,0.2) -- node [above,sloped,scale=0.8] {$S$} (2.7,-1.1);
\draw[->] (1.1,0.1) -- node [above,sloped,scale=0.8] {$S$} (2.7,-2.5);
\draw[->] (4.7,1.7) -- (5.8,1.7);
\draw[->] (4.7,0.3) -- (5.8,1.5);
\draw[->] (4.7,-1.1) -- (5.8,1.3);
\draw[->] (4.7,-2.4) -- node [above,sloped,scale=0.8] {$S$} (6.7,0.2);
\draw[->] (4.7,-2.5) -- node [above,sloped,scale=0.8,near end] {$S$} (6.7,-1.2);
\draw[->] (4.7,-2.6) -- node [above,scale=0.8] {$S$} (6.7,-2.6);
\draw[-,line width=6pt,draw=white] (4.7,1.6) -- (6.7,0.3);
\draw[->] (4.7,1.6) -- node [above,sloped,scale=0.8,near end] {$-S$} (6.7,0.3);
\draw[-,line width=6pt,draw=white] (4.7,0.2) -- (6.7,-1.1);
\draw[->] (4.7,0.2) -- node [above,sloped,scale=0.8] {$-S$} (6.7,-1.1);
\draw[-,line width=6pt,draw=white] (4.7,-1.2) -- (6.7,-2.5);
\draw[->] (4.7,-1.2) -- node [above,sloped,scale=0.8,near end] {$-S$} (6.7,-2.5);
\node at (6.3,1.7) {$\cdots$};
\node at (4.3,0) {$\{1\}$};
\node at (4.3,1.4) {$\{1\}$};
\node at (4.3,-1.4) {$\{1\}$};
\node at (4.3,-2.8) {$\{1\}$};
\node at (8.3,0) {$\{2\}$};
\node at (8.3,-1.4) {$\{2\}$};
\node at (8.3,-2.8) {$\{2\}$};
\end{tikzpicture}
\]
One of the new homological degree $2$ generators, in fact, the one corresponding to $i=2$, can be delooped, and the one with the $-1$-shifted quantum degree can then be cancelled with the new homological degree $1$ generator. After this Gaussian elimination, the complex starts again as in (\ref{eq:genericstart}).

This works all the way until we reach $\CK^\ast((\sigma_1\sigma_2\sigma_3)^4)$. When we tensor this with $\CK^\ast(\sigma_1)$, we also close the two leftmost endpoints of the braid. This means, it actually starts with
\[
\begin{tikzpicture}
\fstrand{0}{0}{0.6}{0.6}
\lcapcup{3}{1.4}{0.6}{0.3}
\fstrand{3}{1.7}{0.6}{0.3}
\mcapcup{3}{0}{0.6}{0.3}
\fstrand{3}{0.3}{0.6}{0.3}
\rcapcup{3}{-1.4}{0.6}{0.3}
\fstrand{3}{-1.1}{0.6}{0.3}
\fstrand{3}{-2.8}{0.6}{0.3}
\lcapcup{3}{-2.5}{0.6}{0.3}
\lcapcup{7}{0}{0.6}{0.3}
\lcapcup{7}{0.3}{0.6}{0.3}
\mcapcup{7}{-1.4}{0.6}{0.3}
\lcapcup{7}{-1.1}{0.6}{0.3}
\rcapcup{7}{-2.8}{0.6}{0.3}
\lcapcup{7}{-2.5}{0.6}{0.3}
\hookl{0}{0}{0.6}{0.6}
\hookl{3}{1.4}{0.6}{0.6}
\hookl{3}{0}{0.6}{0.6}
\hookl{3}{-1.4}{0.6}{0.6}
\hookl{3}{-2.8}{0.6}{0.6}
\hookl{7}{0}{0.6}{0.6}
\hookl{7}{-1.4}{0.6}{0.6}
\hookl{7}{-2.8}{0.6}{0.6}
\draw[->] (1.1,0.4) -- node [above,sloped,scale=0.8] {$S$} (2.7,1.7);
\draw[->] (1.1,0.3) -- node [above,scale=0.8] {$S$} (2.7,0.3);
\draw[->] (1.1,0.2) -- node [above,sloped,scale=0.8] {$S$} (2.7,-1.1);
\draw[->] (1.1,0.1) -- node [above,sloped,scale=0.8] {$S$} (2.7,-2.5);
\draw[->] (4.7,1.7) -- (5.8,1.7);
\draw[->] (4.7,0.3) -- (5.8,1.5);
\draw[->] (4.7,-1.1) -- (5.8,1.3);
\draw[->] (4.7,-2.4) -- node [above,sloped,scale=0.8] {$S$} (6.7,0.2);
\draw[->] (4.7,-2.5) -- node [above,sloped,scale=0.8,near end] {$S$} (6.7,-1.2);
\draw[->] (4.7,-2.6) -- node [above,scale=0.8] {$S$} (6.7,-2.6);
\draw[-,line width=6pt,draw=white] (4.7,1.6) -- (6.7,0.3);
\draw[->] (4.7,1.6) -- node [above,sloped,scale=0.8,near end] {$-S$} (6.7,0.3);
\draw[-,line width=6pt,draw=white] (4.7,0.2) -- (6.7,-1.1);
\draw[->] (4.7,0.2) -- node [above,sloped,scale=0.8] {$-S$} (6.7,-1.1);
\draw[-,line width=6pt,draw=white] (4.7,-1.2) -- (6.7,-2.5);
\draw[->] (4.7,-1.2) -- node [above,sloped,scale=0.8,near end] {$-S$} (6.7,-2.5);
\node at (6.3,1.7) {$\cdots$};
\node at (4.3,0) {$\{1\}$};
\node at (4.3,1.4) {$\{1\}$};
\node at (4.3,-1.4) {$\{1\}$};
\node at (4.3,-2.8) {$\{1\}$};
\node at (8.3,0) {$\{2\}$};
\node at (8.3,-1.4) {$\{2\}$};
\node at (8.3,-2.8) {$\{2\}$};
\end{tikzpicture}
\]
We can still cancel the fourth homological degree $1$ generator with the delooped homological degree $2$ generator. This creates some zigzags starting from the first homological degree $1$ generator. However, we can deloop the homological degree $0$ generator, and one of the new generators cancels the first homological degree $1$ generator. The remaining two homological degree $1$ generators can be delooped, and the $+1$-shifted version cancelled with a homological degree $2$ generator.

The result is
\[
\begin{tikzpicture}
\tstrand{0}{0}{0.6}{0.6}
\tlcapcup{3}{0.7}{0.6}{0.6}
\trcapcup{3}{-0.7}{0.6}{0.6}
\draw[->] (1.4,0.4) -- node [above,sloped,scale=0.8] {$S$} (2.7,1);
\draw[->] (1.4,0.2) -- node [above,sloped,scale=0.8] {$S$} (2.7,-0.4);
\draw[->] (4,1) -- (4.5,1);
\draw[->] (4,-0.4) -- (4.5,-0.4);
\node at (5,0.3) {$\cdots$};
\node at (1,0) {$\{-1\}$};
\end{tikzpicture}
\]
Tensoring with $\CK^\ast(\sigma_2)$, and closing the two leftmost endpoints of the braid allows us to deloop and cancel as in the previous step, until we get
\[
\begin{tikzpicture}
\strand{0}{0}{0.6}
\strand{0.3}{0}{0.6}
\capcup{2.5}{0}{0.6}{0.6}
\draw[->] (1.2,0.3) -- node [above,scale=0.8] {$S$} (2.3,0.3);
\draw[->] (3.7,0.3) -- (4.7,0.3);
\node at (5.2,0.3) {$\cdots$};
\node at (0.8,0) {$\{-2\}$};
\node at (3.3,0) {$\{-1\}$};
\end{tikzpicture}
\]
Closing the remaining braid gives an $R-R$ bimodule complex starting with
\[
\begin{tikzpicture}
\node at (0,0) {$R\otimes R\{-2\}$};
\node at (3,0) {$R\{-1\}$};
\node at (5,0) {$\cdots$};
\draw[->] (1,0) -- node [above] {$\mu$} (2.4,0);
\draw[->] (3.6,0) -- node [above] {$\varepsilon$} (4.6,0);
\end{tikzpicture}
\]
Since $\mu$ is surjective and this is a cochain complex, we get $\varepsilon=0$. After a quantum shift involving $n(n+2)-1$ positive crossings, the result follows for $n=3$.

For arbitrary $n$ we observe that we get $n$ generators in homological degree $1$ in (\ref{eq:genericstart}), which remains true up to the braid word $(\sigma_1\cdots \sigma_n)^{n+1}$. With every letter in the remaining word $\sigma_1\cdots\sigma_{n-1}$ we get one less generator in homological degree $1$ just as above.
\end{proof}

\begin{remark}\label{rem:extraD}
For $n\leq 7$ computer calculations show that
\[
\CK^\ast(L_n) \simeq E^\ast\{n(n+1)\} \oplus D^\ast(2)[2]\{n(n+1)+3\}\oplus H^\ast
\]
with $H^\ast$ an $R-R$ bimodule complex concentrated in homological degrees bigger than $3$. We believe this to be true for general $n$, and consider the stability results of \cite{MR2308944} as supporting evidence. Attempting to prove this along the current arguments seems to be somewhat tedious though.

But the existence of a $D^\ast(2)$ direct summand in the chain homotopy type of $\CK^\ast(L_3)$ shows that we can replace the trefoil by $L_3$ in Theorem \ref{thm:maintheorem}. Indeed, we can replace the trefoil by any link which has a $D^\ast(2)$ direct summand in the chain homotopy type of its Khovanov complex.
\end{remark}

\begin{proof}[Proof of Theorem \ref{thm:maintheorem}]
Up to chain homotopy, $\CK^\ast(L_3^k;\Z_{(3)})$ contains a direct summand
\[
 (C^\ast(3;\R)[8]\{25\})^{\otimes l} \otimes_{\R} (E_3^\ast\{12\})^{\otimes k-l} \otimes_{\R} D^\ast(2;\R)[2]\{7\}
\]
for every $l\in \{1,\ldots,k\}$ by Lemma \ref{lem:mainlemma}, Lemma \ref{lem:beginner}, (\ref{eq:connectedsum}) and (\ref{eq:T23complex}). Here we use the notation $E_3^\ast = E^\ast \otimes \Z_{(3)}$.

By Lemma \ref{lem:removeE} such a direct summand is chain homotopy equivalent to
\begin{equation}\label{eq:tensorCCD}
 C^\ast(3;\R)^{\otimes l-1} \otimes_{\R} C^\ast(\pm 3;\R)\otimes_{\R} D^\ast(2;\R)[8l+2]\{11k+14l+7\}.
\end{equation}
By Lemma \ref{lem:tensorC} and Lemma \ref{lem:tensorCD} we get plenty of direct summands
\[
 D^\ast(2\cdot 3^l;\R)
\]
suitably shifted. Each of these direct summands creates a $\Z_{(3)}/2\cdot3^l\Z_{(3)}\cong \Z/3^l\Z$ direct summand in $H_{\mathrm{Kh}}^\ast(L_3;\Z_{(3)})$. Since $\Z_{(3)}$ is a localization of $\Z$,
\[
 H_{\mathrm{Kh}}^\ast(L_3;\Z_{(3)}) \cong H_{\mathrm{Kh}}^\ast(L_3;\Z) \otimes \Z_{(3)},
\]

and these direct summands have to be already present in $H_{\mathrm{Kh}}^\ast(L_3;\Z)$.
\end{proof}

\begin{remark}
From (\ref{eq:tensorCCD}) we can work out some of the bidegrees where $3^l$-torsion occurs. To get a direct summand $D^\ast(2\cdot 3^l;\R)$, we need to apply Lemma \ref{lem:tensorC} $(l-1)$-times, and Lemma \ref{lem:tensorCD} once.  If we only focus on minimal homological degree in these lemmas, we get a direct summand
\[
 D^\ast(2\cdot 3^l;\R)[8l+2]\{11k+12l+7\},
\]
and by focussing on maximal homological degree we get a direct summand
\[
 D^\ast(2\cdot 3^l;\R)[9l+2]\{11k+16l+7\}.
\]
Given that $D^\ast(n)$ has $n$-torsion in bidegree $(1,0)$, we get a summand $\Z/3^l\Z$ in the Khovanov cohomology of $L^k_3$ in bidegrees 
\[
(8l+3,11k+12l+7) \mbox{ and }(9l+3,11k+16l+7).
\]
Analyzing Lemma \ref{lem:tensorC} and Lemma \ref{lem:tensorCD} a bit more carefully, we see that in bidegree 
\[
 (8l+3+m,11k+12l+7+4m)
\]
there are at least $\begin{pmatrix} l \\ m \end{pmatrix}$ copies of $\Z/3^l\Z$ for $m=0,\ldots,l$.

For $l<k$ there exist more direct summands of $\Z/3^l\Z$ in the Khovanov cohomology of $L_3^k$. In view of Remark \ref{rem:extraD} this is not surprising. But for $l=k$ calculations up to $k=6$ have found these to be all the direct summands of $\Z/3^k\Z$.
\end{remark}

Computer calculations show that
\begin{equation}\label{eq:M5}
 \CK^\ast(L_5;\Z_{(5)}) \simeq F_5^\ast \oplus C^\ast(5;R_5)[18]\{55\}
\end{equation}
where $F_5^\ast$ is a finitely generated $R_5-R_5$ bimodule complex concentrated in homological degrees $0$ to $17$, and
\begin{equation}\label{eq:M7}
 \CK^\ast(L_7;\Z_{(7)}) \simeq F_7^\ast \oplus C^\ast(7;R_7)[32]\{97\}
\end{equation}
where $F_7^\ast$ is a finitely generated $R_7-R_7$ bimodule complex concentrated in homological degrees $0$ to $32$. With the same arguments as in the proof of Theorem \ref{thm:maintheorem} this confirms Conjecture \ref{con:mainconj} for $p=5$ and $p=7$.

\begin{remark}
From the proof of Lemma \ref{lem:mainlemma} it seems unlikely that the analogous statement works for $\Z$ coefficients. However, we really only needed to invert $2$. Similarly, for (\ref{eq:M5}) and (\ref{eq:M7}) we only need to invert the primes $2$ and $3$.

In particular, we can form a `mixed' link
\[
 L = L_5^{\# r} \co L_7^{\# s} \co T(2,3)
\]
which has direct summands $\Z/ 5^r7^s \Z$ in several single bidegrees of its Khovanov cohomology.
\end{remark}

Obtaining (\ref{eq:M5}) and (\ref{eq:M7}) by hand seems daunting, but may not be impossible. A general technique may also work for other odd primes. It is encouraging that the direct summand is at the top in terms of supported homological degrees. We note however that $L_5$ has $3$-torsion in homological degree $20$, and $L_7$ has $2$- and $3$-torsion in homological degree $34$, that is, above the homological support of the localized versions.

Nevertheless we refine Conjecture \ref{con:mainconj} to
\begin{conjecture}
Let $p$ be an odd prime. Then $\CK^\ast(L_p;\Z_{(p)})$, viewed as a $R_p-R_p$ bimodule complex contains $C^\ast(p;R_p)$ suitably shifted as a direct summand up to chain homotopy.
\end{conjecture}

\begin{remark}
One can ask whether connected sums of knots can increase the order of torsion in Khovanov cohomology. Indeed, in \cite{2019arXiv190606278M} the first author observed that the connected sum of $T(5,6)$ with itself gives rise to torsion of order $9$. However, a connected sum of three or four $T(5,6)$ does not give rise to torsion of order greater than $9$.

We can consider $\CK^\ast(K)$ as an $R-R$ bimodule by placing two basepoints on $K$, but since the connected sum of knots does not depend on where the basepoint sits, this bimodule structure cannot be as asymmetric as in the case of $L_3$. While a summand $C^\ast(p;\Z_{(p)})$ could be balanced by another summand which flips the $R-R$ bimodule structure, it seems unlikely to get examples with that.
\end{remark}

\appendix

\section{A hands-on proof of Lemma \ref{lem:mainlemma}}
\label{app:handson}
The purpose of this appendix is to give stages in the Bar-Natan algorithm which lead to (\ref{eq:l3ending}). Instead of scanning each crossing one can use the original divide-and-conquer approach of \cite{MR2320156}. Furthermore, we are only interested in the higher homological degrees of the final complex, so in the later steps we can ignore lower homological degrees.
There are still a lot of cancellations required and we do not give every detail. In the cases of $L_5$ and $L_7$ it seems hopeless trying to write down the steps.

For the first three crossings we can never deloop or cancel, and $\CK^\ast(\sigma_1\sigma_2\sigma_3)$ is as in Figure \ref{fig:firstthree}.
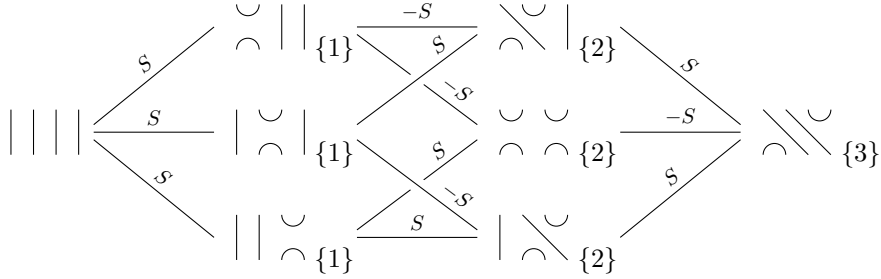
\begin{figure}[ht]
\begin{center}
\begin{tikzpicture}
\fstrand{0}{0}{0.6}{0.6}
\lcapcup{3}{1.4}{0.6}{0.6}
\mcapcup{3}{0}{0.6}{0.6}
\rcapcup{3}{-1.4}{0.6}{0.6}
\lbslash{6.5}{1.4}{0.6}{0.6}
\dcapcup{6.5}{0}{0.6}{0.6}
\rbslash{6.5}{-1.4}{0.6}{0.6}
\dbslash{10}{0}{0.6}{0.6}
\draw[-] (1.1,0.4) -- node[above,sloped,scale=0.8] {$S$} (2.7,1.7);
\draw[-] (1.1,0.3) -- node[above,scale=0.8] {$S$} (2.7,0.3);
\draw[-] (1.1,0.2) -- node[above,sloped,scale=0.8] {$S$} (2.7,-1.1);
\draw[-] (4.6,1.7) -- node[above,scale=0.8] {$-S$} (6.2,1.7);
\draw[-] (4.6,1.6) -- node[above,scale=0.8,sloped,near end] {$-S$} (6.2,0.4);
\draw[-] (4.6,-1.1) -- node[above,scale=0.8] {$S$} (6.2,-1.1);
\draw[-] (4.6,-1) -- node[above,sloped,scale=0.8,near end] {$S$} (6.2,0.2);
\draw[-,line width=6pt,draw=white] (4.6,0.4) -- (6.2,1.6);
\draw[-,line width=6pt,draw=white] (4.6,0.2) -- (6.2,-1);
\draw[-] (4.6,0.4) -- node[above,sloped,scale=0.8,near end] {$S$} (6.2,1.6);
\draw[-] (4.6,0.2) -- node[above,sloped,scale=0.8,near end] {$-S$} (6.2,-1);
\draw[-] (8.1,1.7) -- node[above,sloped,scale=0.8] {$S$} (9.7,0.4);
\draw[-] (8.1,0.3) -- node[above,sloped,scale=0.8] {$-S$} (9.7,0.3);
\draw[-] (8.1,-1.1) -- node[above,sloped,scale=0.8] {$S$} (9.7,0.2);
\node at (4.3,0) {$\{1\}$};
\node at (4.3,1.4) {$\{1\}$};
\node at (4.3,-1.4) {$\{1\}$};
\node at (7.8,1.4) {$\{2\}$};
\node at (7.8,0) {$\{2\}$};
\node at (7.8,-1.4) {$\{2\}$};
\node at (11.3,0) {$\{3\}$};
\end{tikzpicture}
\caption{\label{fig:firstthree} The cochain complex $\CK^\ast(\sigma_1\sigma_2\sigma_3)$.}
\end{center}
\end{figure}

We now form $\CK^\ast(\sigma_1\sigma_2\sigma_3)\otimes \CK^\ast(\sigma_1\sigma_2\sigma_3)$, and begin with the delooping and cancelling. It turns out that the generator in homological degree $6$ can be cancelled, and all other generators in homological degree $5$ can also be cancelled. Indeed, only two generators in homological degree $4$ survive. As in Lemma \ref{lem:beginner} we can reduce the number of homological degree $1$ generators to three.

The resulting cochain complex $C^\ast$ is depicted in Figure \ref{fig:aftersix}.
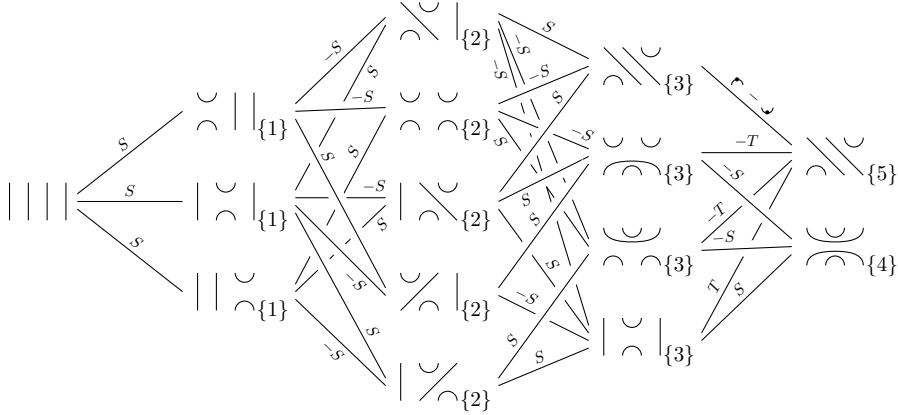
\begin{figure}
\begin{center}
\begin{tikzpicture}
\fstrand{0}{0}{0.5}{0.5}
\lcapcup{2.5}{1.2}{0.5}{0.5}
\mcapcup{2.5}{0}{0.5}{0.5}
\rcapcup{2.5}{-1.2}{0.5}{0.5}
\lbslash{5.2}{2.4}{0.5}{0.5}
\dcapcup{5.2}{1.2}{0.5}{0.5}
\rbslash{5.2}{0}{0.5}{0.5}
\lslash{5.2}{-1.2}{0.5}{0.5}
\rslash{5.2}{-2.4}{0.5}{0.5}
\dbslash{7.9}{1.8}{0.5}{0.5}
\dlcap{7.9}{0.6}{0.5}{0.5}
\ducup{7.9}{-0.6}{0.5}{0.5}
\mcapcup{7.9}{-1.8}{0.5}{0.5}
\dbslash{10.6}{0.6}{0.5}{0.5}
\dulcp{10.6}{-0.6}{0.5}{0.5}
\node[scale=0.8] at (3.5,1.2) {$\{1\}$};
\node[scale=0.8] at (3.5,0) {$\{1\}$};
\node[scale=0.8] at (3.5,-1.2) {$\{1\}$};
\node[scale=0.8] at (6.2,2.4) {$\{2\}$};
\node[scale=0.8] at (6.2,1.2) {$\{2\}$};
\node[scale=0.8] at (6.2,0) {$\{2\}$};
\node[scale=0.8] at (6.2,-1.2) {$\{2\}$};
\node[scale=0.8] at (6.2,-2.4) {$\{2\}$};
\node[scale=0.8] at (8.9,1.8) {$\{3\}$};
\node[scale=0.8] at (8.9,0.6) {$\{3\}$};
\node[scale=0.8] at (8.9,-0.6) {$\{3\}$};
\node[scale=0.8] at (8.9,-1.8) {$\{3\}$};
\node[scale=0.8] at (11.6,0.6) {$\{5\}$};
\node[scale=0.8] at (11.6,-0.6) {$\{4\}$};
\draw[-] (0.9,0.35) -- node [above,sloped,scale=0.6] {$S$} (2.3,1.45);
\draw[-] (0.9,0.25) -- node [above,sloped,scale=0.6] {$S$} (2.3,0.25);
\draw[-] (0.9,0.15) -- node [above,sloped,scale=0.6] {$S$} (2.3,-0.95);
\draw[-] (3.8,-0.95) -- node [below,sloped,scale=0.6,very near end] {$S$} (5,0.2);
\draw[-] (3.8,-0.85) -- node [above,sloped,scale=0.6,near end] {$S$} (5,1.4);
\draw[-,line width=6pt,draw=white] (3.8,0.1) -- (5,-2.1);
\draw[-,line width=6pt,draw=white] (3.8,0.2) -- (5,-1);
\draw[-] (3.8,0.3) -- node [above,scale=0.6,very near end] {$-S$} (5,0.3);
\draw[-] (3.8,0.4) -- node [below,sloped,scale=0.6,near end] {$S$} (5,2.6);
\draw[-,line width=6pt,draw=white] (3.8,1.35) -- (5,-0.9);
\draw[-] (3.8,1.35) -- node [above,sloped,scale=0.6,near start] {$S$} (5,-0.9);
\draw[-,line width=6pt,draw=white] (3.8,1.45) -- (5,1.5);
\draw[-] (3.8,1.45) -- node [above,sloped,scale=0.6,near end] {$-S$} (5,1.5);
\draw[-] (3.8,1.55) -- node [above,sloped,scale=0.6] {$-S$} (5,2.7);
\draw[-] (3.8,-1.05) -- node [below,sloped,scale=0.6] {$-S$}(5,-2.2);
\draw[-] (3.8,0.1) -- node [above,sloped,scale=0.6,near end] {$S$} (5,-2.1);
\draw[-] (3.8,0.2) -- node [below,sloped,scale=0.6,near end] {$-S$} (5,-1);
\draw[-] (6.5,2.55) -- node [below,sloped,scale=0.6,very near start] {$-S$} (7.7,-1.4);
\draw[-] (6.5,2.65) -- node [above,sloped,scale=0.6,very near start] {$-S$} (7.7,-0.25);
\draw[-] (6.5,2.75) -- node [above,sloped,scale=0.6] {$S$} (7.7,2.15);
\draw[-,line width=6pt,draw=white] (6.5,1.55) -- (7.7,2.05);
\draw[-] (6.5,1.55) -- node [above,sloped,scale=0.6] {$-S$} (7.7,2.05);
\draw[-,line width=6pt,draw=white] (6.5,1.45) -- (7.7,0.95);
\draw[-] (6.5,1.45) -- node [above,sloped,scale=0.6,very near end] {$-S$} (7.7,0.95);
\draw[-] (6.5,1.35) -- node [below,sloped,scale=0.6,very near start] {$S$} (7.7,-0.35);
\draw[-,line width=6pt,draw=white] (6.5,0.35) -- (7.7,1.95);
\draw[-] (6.5,0.35) -- node [above,sloped,scale=0.6,near end] {$S$} (7.7,1.95);
\draw[-,line width=6pt,draw=white] (6.5,0.25) -- (7.7,0.85);
\draw[-] (6.5,0.25) -- node [below,sloped,scale=0.6,near start] {$S$} (7.7,0.85);
\draw[-] (6.5,0.15) -- node [above,sloped,scale=0.6] {$S$} (7.7,-1.5);
\draw[-,line width=6pt,draw=white] (6.5,-0.9) -- (7.7,0.75);
\draw[-] (6.5,-0.9) -- node [above,sloped,scale=0.6] {$S$} (7.7,0.75);
\draw[-] (6.5,-1) -- node [above,sloped,scale=0.6,near start] {$-S$} (7.7,-1.6);
\draw[-,line width=6pt,draw=white] (6.5,-2.1) -- (7.7,-0.45);
\draw[-] (6.5,-2.1) -- node [above,sloped,scale=0.6,near start] {$S$} (7.7,-0.45);
\draw[-] (6.5,-2.2) -- node [above,sloped,scale=0.6] {$S$} (7.7,-1.7);
\draw[-] (9.2,-1.6) -- node [above,sloped,scale=0.6] {$S$} (10.4,-0.45);
\draw[-] (9.2,-1.5) -- node [above,sloped,scale=0.6,near start] {$T$} (10.4,0.7);
\draw[-,line width=6pt,draw=white] (9.2,-0.4) -- (10.4,-0.35);
\draw[-] (9.2,-0.4) -- node [above,sloped,scale=0.6,near start] {$-S$} (10.4,-0.35);
\draw[-] (9.2,-0.3) -- node [above,sloped,scale=0.6,near start] {$-T$} (10.4,0.8);
\draw[-,line width=6pt,draw=white] (9.2,0.8) -- (10.4,-0.25);
\draw[-] (9.2,0.8) -- node [above,sloped,scale=0.6,near start] {$-S$} (10.4,-0.25);
\draw[-] (9.2,0.9) -- node [above,sloped,scale=0.6] {$-T$} (10.4,0.9);
\draw[-] (9.2,2.05) -- node [above,sloped,scale=0.6] {$-$} (10.4,1);
\jcap{9.6}{1.8}{0.3}{0.3}
\jcup{10}{1.2}{0.3}{0.3}
\node[scale=0.4] at (9.675,1.87) {$\bullet$};
\node[scale=0.4] at (10.075,1.43) {$\bullet$};
\end{tikzpicture}
\caption{\label{fig:aftersix} The cochain complex $C^\ast$ chain homotopy equivalent to $\CK^\ast((\sigma_1\sigma_2\sigma_3)^2)$. The morphism $T$ stands for two surgeries.}
\end{center}
\end{figure}

We now need to form $C^\ast\otimes C^\ast$. Obviously, this has a lot of generators. For the next steps we only need the top half of the cochain complex after cancellations. In Figure \ref{fig:bigend} we show the generators in homological degrees $6$ to $8$. A few more of the homological degree $5$ generators are needed for later cancellations, but not all. We omit the details.
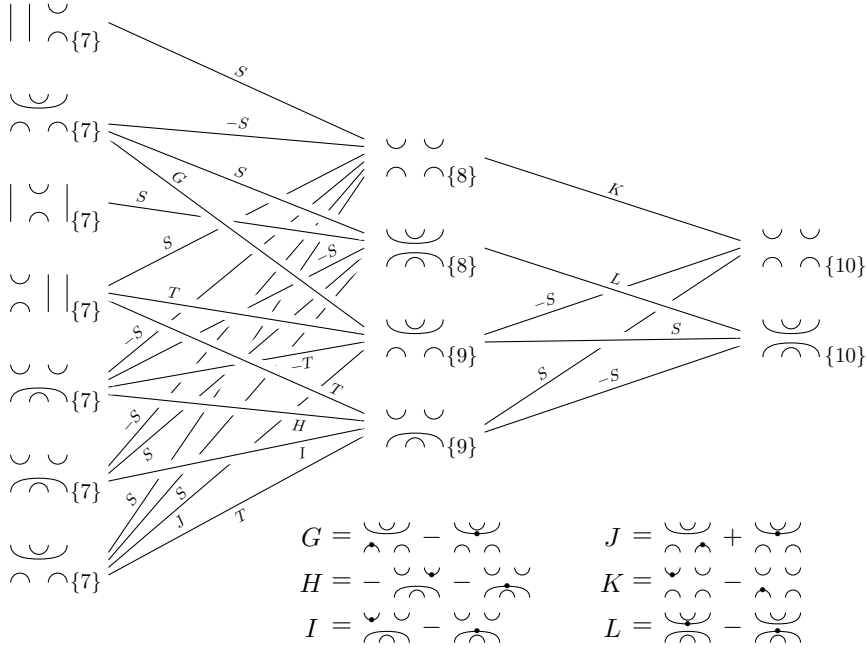
\begin{figure}[ht]
\begin{center}
\begin{tikzpicture}
\rcapcup{0}{3.6}{0.5}{0.5}
\ducup{0}{2.4}{0.5}{0.5}
\mcapcup{0}{1.2}{0.5}{0.5}
\lcapcup{0}{0}{0.5}{0.5}
\dlcap{0}{-1.2}{0.5}{0.5}
\dlcap{0}{-2.4}{0.5}{0.5}
\ducup{0}{-3.6}{0.5}{0.5}
\dcapcup{5}{1.8}{0.5}{0.5}
\dulcp{5}{0.6}{0.5}{0.5}
\ducup{5}{-0.6}{0.5}{0.5}
\dlcap{5}{-1.8}{0.5}{0.5}
\dcapcup{10}{0.6}{0.5}{0.5}
\dulcp{10}{-0.6}{0.5}{0.5}
\draw[-] (1.3,-3.2) -- node [above,sloped,scale=0.6,very near start] {$S$} (4.7,1.8);
\draw[-] (1.3,-3.3) -- node [below,sloped,scale=0.6,near start] {$S$} (4.7,0.65);
\draw[-] (1.3,-3.4) -- node [below,sloped,scale=0.6,near start] {$J$} (4.7,-0.5);
\draw[-] (1.3,-3.5) -- node [below,sloped,scale=0.6] {$T$} (4.7,-1.65);
\draw[-,line width=6pt,draw=white] (1.3,-2.25) -- (4.7,-1.55);
\draw[-] (1.3,-2.25) -- node [below,sloped,scale=0.6,near end] {$I$} (4.7,-1.55);
\draw[-,line width=6pt,draw=white] (1.3,-2.15) -- (4.7,0.75);
\draw[-] (1.3,-2.15) -- node [below,sloped,scale=0.6,very near start] {$S$} (4.7,0.75);
\draw[-] (1.3,-2.05) -- node [above,sloped,scale=0.6,very near start] {$-S$} (4.7,1.9);
\draw[-,line width=6pt,draw=white] (1.3,-1.1) -- (4.7,-1.45);
\draw[-] (1.3,-1.1) -- node [below,sloped,scale=0.6,near end] {$H$} (4.7,-1.45);
\draw[-,line width=6pt,draw=white] (1.3,-1) -- (4.7,-0.4);
\draw[-] (1.3,-1) -- node [below,sloped,scale=0.6,near end] {$-T$}(4.7,-0.4);
\draw[-,line width=6pt,draw=white] (1.3,-0.9) -- (4.7,0.85);
\draw[-] (1.3,-0.9) -- node [above,sloped,scale=0.6,very near end] {$-S$} (4.7,0.85);
\draw[-] (1.3,-0.8) -- node [above,sloped,scale=0.6,very near start] {$-S$} (4.7,2);
\draw[-,line width=6pt,draw=white] (1.3,0.15) -- (4.7,-1.35);
\draw[-] (1.3,0.15) -- node [above,sloped,scale=0.6,very near end] {$T$} (4.7,-1.35);
\draw[-,line width=6pt,draw=white] (1.3,0.25) -- (4.7,-0.3);
\draw[-] (1.3,0.25) -- node [above,sloped,scale=0.6,near start] {$T$} (4.7,-0.3);
\draw[-] (1.3,0.35) -- node [above,sloped,scale=0.6,near start] {$S$} (4.7,2.1);
\draw[-,line width=6pt,draw=white] (1.3,1.45) -- (4.7,0.95);
\draw[-] (1.3,1.45) -- node [above,sloped,scale=0.6,very near start] {$S$} (4.7,0.95);
\draw[-,line width=6pt,draw=white] (1.3,2.3) -- (4.7,-0.2);
\draw[-] (1.3,2.3) -- node [above,sloped,scale=0.6,near start] {$G$} (4.7,-0.2);
\draw[-,line width=6pt,draw=white] (1.3,2.4) -- (4.7,1.05);
\draw[-] (1.3,2.4) -- node [above,sloped,scale=0.6] {$S$} (4.7,1.05);
\draw[-] (1.3,2.5) -- node [above,sloped,scale=0.6] {$-S$} (4.7,2.2);
\draw[-] (1.3,3.85) -- node [above,sloped,scale=0.6] {$S$} (4.7,2.3);
\draw[-] (6.3,-1.6) -- node [above,sloped,scale=0.6] {$-S$} (9.7,-0.45);
\draw[-] (6.3,-1.5) -- node [above,sloped,scale=0.6,near start] {$S$} (9.7,0.75);
\draw[-,line width=6pt,draw=white] (6.3,-0.4) -- (9.7,-0.35);
\draw[-] (6.3,-0.4) -- node [above,sloped,scale=0.6,near end] {$S$} (9.7,-0.35);
\draw[-] (6.3,-0.3) -- node [above,sloped,scale=0.6,near start] {$-S$} (9.7,0.85);
\draw[-,line width=6pt,draw=white] (6.3,0.85) -- (9.7,-0.25);
\draw[-] (6.3,0.85) -- node [above,sloped,scale=0.6] {$L$} (9.7,-0.25);
\draw[-] (6.3,2.05) -- node [above,sloped,scale=0.6] {$K$} (9.7,0.95);
\node[scale=0.8] at (1,3.6) {$\{7\}$};
\node[scale=0.8] at (1,2.4) {$\{7\}$};
\node[scale=0.8] at (1,1.2) {$\{7\}$};
\node[scale=0.8] at (1,0) {$\{7\}$};
\node[scale=0.8] at (1,-1.2) {$\{7\}$};
\node[scale=0.8] at (1,-2.4) {$\{7\}$};
\node[scale=0.8] at (1,-3.6) {$\{7\}$};
\node[scale=0.8] at (6,1.8) {$\{8\}$};
\node[scale=0.8] at (6,0.6) {$\{8\}$};
\node[scale=0.8] at (6,-0.6) {$\{9\}$};
\node[scale=0.8] at (6,-1.8) {$\{9\}$};
\node[scale=0.8] at (11.1,0.6) {$\{10\}$};
\node[scale=0.8] at (11.1,-0.6) {$\{10\}$};
\node at (4,-3) {$G$};
\node at (4,-3.6) {$H$};
\node at (4,-4.2) {$I$};
\node at (8,-3) {$J$};
\node at (8,-3.6) {$K$};
\node at (8,-4.2) {$L$};
\node at (4.4,-3) {$=$};
\node at (4.4,-3.6) {$=$};
\node at (4.4,-4.2) {$=$};
\node at (8.4,-3) {$=$};
\node at (8.4,-3.6) {$=$};
\node at (8.4,-4.2) {$=$};
\node at (5.6,-3) {$-$};
\node at (4.8,-3.6) {$-$};
\node at (6,-3.6) {$-$};
\node at (5.6,-4.2) {$-$};
\node at (9.6,-3) {$+$};
\node at (9.6,-3.6) {$-$};
\node at (9.6,-4.2) {$-$};
\node[scale=0.5] at (4.8,-3.1) {$\bullet$};
\node[scale=0.5] at (6.2,-2.95) {$\bullet$};
\node[scale=0.5] at (5.6,-3.5) {$\bullet$};
\node[scale=0.5] at (6.6,-3.65) {$\bullet$};
\node[scale=0.5] at (4.8,-4.1) {$\bullet$};
\node[scale=0.5] at (6.2,-4.25) {$\bullet$};
\node[scale=0.5] at (9.2,-3.1) {$\bullet$};
\node[scale=0.5] at (10.2,-2.95) {$\bullet$};
\node[scale=0.5] at (8.8,-3.5) {$\bullet$};
\node[scale=0.5] at (10,-3.7) {$\bullet$};
\node[scale=0.5] at (9,-4.15) {$\bullet$};
\node[scale=0.5] at (10.2,-4.25) {$\bullet$};
\ducup{4.7}{-3.2}{0.4}{0.4}
\ducup{5.9}{-3.2}{0.4}{0.4}
\dlcap{5.1}{-3.8}{0.4}{0.4}
\dlcap{6.3}{-3.8}{0.4}{0.4}
\dlcap{4.7}{-4.4}{0.4}{0.4}
\dlcap{5.9}{-4.4}{0.4}{0.4}
\ducup{8.7}{-3.2}{0.4}{0.4}
\ducup{9.9}{-3.2}{0.4}{0.4}
\dcapcup{8.7}{-3.8}{0.4}{0.4}
\dcapcup{9.9}{-3.8}{0.4}{0.4}
\dulcp{8.7}{-4.4}{0.4}{0.4}
\dulcp{9.9}{-4.4}{0.4}{0.4}
\end{tikzpicture}
\end{center}
\caption{\label{fig:bigend}The cochain complex $D^\ast$, chain homotopy equivalent to $\CK^\ast((\sigma_1\sigma_2\sigma_3)^4)$, in homological degrees $6$ to $8$. The morphism $T$ stands for two surgeries.}
\end{figure}

We now form $D^\ast\otimes \CK^\ast(\sigma_1)$. After delooping and cancellations, we get a cochain complex $E^\ast$ ending in homological degree $9$ as in Figure \ref{fig:atlast}.
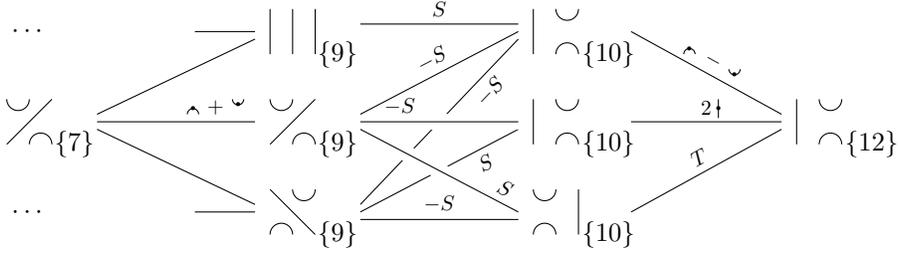
\begin{figure}[ht]
\begin{center}
\begin{tikzpicture}
\tstrand{0}{2.4}{0.6}{0.6}
\slashcps{0}{1.2}{0.6}{0.6}
\bslashcps{0}{0}{0.6}{0.6}
\trcapcup{3.5}{2.4}{0.6}{0.6}
\trcapcup{3.5}{1.2}{0.6}{0.6}
\tlcapcup{3.5}{0}{0.6}{0.6}
\trcapcup{7}{1.2}{0.6}{0.6}
\slashcps{-3.5}{1.2}{0.6}{0.6}
\node at (-3.2,2.7) {$\cdots$};
\node at (-3.2,0.3) {$\cdots$};
\draw[-] (-2.3,1.4) -- (-0.2,0.4);
\draw[-] (-2.3,1.5) -- node [above,scale=0.8,near end] {$+$} (-0.2,1.5);
\draw[-] (-2.3,1.6) -- (-0.2,2.6);
\draw[-] (-1,2.7) -- (-0.2,2.7);
\draw[-] (-1,0.3) -- (-0.2,0.3);
\draw[-] (1.2,0.2) -- node [above,scale=0.8] {$-S$} (3.3,0.2);
\draw[-] (1.2,0.3) -- node [below,sloped,scale=0.8,near end] {$S$} (3.3,1.4);
\draw[-] (1.2,0.4) -- node [below,sloped,scale=0.8,near end] {$-S$} (3.3,2.6);
\draw[-,line width=6pt,draw=white] (1.2,1.4) -- (3.3,0.3);
\draw[-,line width=6pt,draw=white] (1.2,1.5) -- (3.3,1.5);
\draw[-] (1.2,1.4) -- node [above,sloped,scale=0.8,very near end] {$S$} (3.3,0.3);
\draw[-] (1.2,1.5) -- node [above,scale=0.8,near start] {$-S$} (3.3,1.5);
\draw[-] (1.2,1.6) -- node [above,sloped,scale=0.8] {$-S$} (3.3,2.7);
\draw[-] (1.2,2.8) -- node [above,scale=0.8] {$S$} (3.3,2.8);
\draw[-] (4.8,0.3) -- node [above,sloped,scale=0.8] {$T$} (6.8,1.4);
\draw[-] (4.8,1.5) -- node [above,scale=0.8] {$2$} (6.8,1.5);
\draw[-] (4.8,2.7) -- node [above,sloped,scale=0.8] {$-$} (6.8,1.6);
\node at (-2.6,1.2) {$\{7\}$};
\node at (0.9,2.4) {$\{9\}$};
\node at (0.9,1.2) {$\{9\}$};
\node at (0.9,0) {$\{9\}$};
\node at (4.5,2.4) {$\{10\}$};
\node at (4.5,1.2) {$\{10\}$};
\node at (4.5,0) {$\{10\}$};
\node at (8,1.2) {$\{12\}$};
\jcap{-1.1}{1.6}{0.3}{0.3}
\jcup{-0.5}{1.5}{0.3}{0.3}
\node[scale=0.4] at (-1.025,1.675) {$\bullet$};
\node[scale=0.4] at (-0.425,1.725) {$\bullet$};
\strand{5.96}{1.55}{0.25}
\jcap{5.5}{2.4}{0.3}{0.3}
\jcup{6.1}{1.9}{0.3}{0.3}
\node[scale=0.4] at (5.96,1.675) {$\bullet$};
\node[scale=0.4] at (5.575,2.475) {$\bullet$};
\node[scale=0.4] at (6.175,2.125) {$\bullet$};
\end{tikzpicture}
\caption{\label{fig:atlast} The cochain complex $E^\ast$.}
\end{center}
\end{figure}
We note that it is possible to get the number of homological degree $6$ objects down to $3$, but we only need the one depicted in the next step.

The last step is to form $E^\ast\otimes \CK^\ast(\sigma_2)$, deloop and cancel. Notice that the single generator depicted in homological degree $6$ leads to two generators in homological degree $7$ after delooping, and the one with the larger quantum grading is needed to cancel a generator in homological degree $8$.

This leads to (\ref{eq:l3ending}). Notice that we have a few more generators coming from $E^6\otimes \CK^1(\sigma_2)$ in homological degree $7$, but these do not map to any of the surviving generators in homological degree $8$. It is possible to cancel them with generators in homological degree $6$, but this would require us to keep track of a larger part of the cochain complex $E^\ast$.

\begin{remark}
It is possible to do the cancellations in a different order, which can result in different cochain complexes. The above listed complexes were in fact also obtained by a computer programme, `SKnotJob', written by the second author. The available version of SKnotJob does not have this feature, one has to change a few lines in the source code to get the necessary output.
But in order to interpret this output correctly, a deeper understanding of the programme is necessary. We can provide the interested reader with the various outputs, together with information on how to interpret it.
\end{remark}

\bibliography{KnotHomology}
\bibliographystyle{amsalpha}

\end{document}